\documentclass[12pt,nopagetitle]{article}
\usepackage{amsfonts, amsmath, amssymb, amsthm, enumitem}  
\usepackage[english]{babel}
\usepackage[utf8]{inputenc}
\usepackage[OT1]{fontenc}
\usepackage{bm}

\usepackage{textcomp, cmap, comment}	
\usepackage{graphicx, wrapfig}
\usepackage{epigraph, xcolor, hyperref}
\usepackage{array,tabularx,tabulary,booktabs}
\usepackage{euscript, mathrsfs}	

\newcounter{iii}

\newcommand{\bb}{{\mathcal B}}
\newcommand{\aaa}{{\mathcal A}}

\newcommand{\G}{\mathcal G}

\newcommand{\ff}{\mathcal F}
\theoremstyle{plain}
\newtheorem{thm}{Theorem}
\newtheorem{stat}{Statement}
\newtheorem{lem}{Lemma}
\newtheorem{prop}{Proposition}
\newtheorem{pro}{Problem}
\newtheorem{cor}{Corollary}
\theoremstyle{definition}

\title{Structure and properties of large intersecting families}
\author{Andrey Kupavskii\footnote{Moscow Institute of Physics and Technology, Ecole Polytechnique F\'ed\'erale de Lausanne; Email: {\tt kupavskii@yandex.ru} \ \ Research supported by the grant RNF~16-11-10014.}}
\date{}

\begin{document}
\maketitle
\begin{abstract} We say that a family of $k$-subsets of an $n$-element set is intersecting, if any two of its sets intersect. In this paper we study different extremal properties of intersecting families, as well as the structure of large intersecting families. We also give some results on $k$-uniform families without $s$ pairwise disjoint sets, related to Erd\H os Matching Conjecture.

We prove a conclusive version of Frankl's theorem on intersecting families with bounded maximal degree. This theorem, along with its generalizations to cross-intersecting families, implies many results on the topic, obtained by Frankl, Frankl and Tokushige, Kupavskii and Zakharov and others.

We study the structure of large intersecting families, obtaining some general structural theorems which generalize the results of Han and Kohayakawa, as well as Kostochka and Mubayi.

We give degree and subset degree version of the Erd\H os--Ko--Rado and the Hilton--Milner theorems, extending the results of Huang and Zhao, and Frankl, Han, Huang and Zhao. We also extend the range in which the degree version of the Erd\H os Matching conjecture holds.

\end{abstract}

\section{Introduction}

Let $[n]:=\{1,\ldots, n\}$ and denote $2^{[n]}$ the power set of $[n]$.  For integers $a\le b$ we put $[a,b]:=\{a,a+1,\ldots, b\}$ and for integers $0\le k\le n$ we denote ${[n]\choose k}$ the collection of all $k$-element subsets ({\it $k$-sets}) of $[n]$. Any collection of sets $\ff\subset 2^{[n]}$ we call a family.  We call a family \underline{intersecting}, if any two sets from it intersect. A ``trivial'' example of such family is all sets containing a fixed element. We call a family \underline{non-trivial}, if not all of its sets contain the same element.

One of the oldest and most famous results in extremal combinatorics is the Erd\H os--Ko--Rado theorem:
\begin{thm}[\cite{EKR}] Let $n\ge 2k>0$ and consider an intersecting family $\ff\subset {[n]\choose k}$. Then $|\ff|\le {n-1\choose k-1}$. Moreover, for $n>2k$ the only families attaining this bound are the families of all $k$-sets containing a given element. \end{thm}

Answering a question of Erd\H os, Ko, and Rado, Hilton and Milner \cite{HM} found the largest non-trivial intersecting family of $k$-sets. Up to permutation of the ground sets, it is the family consisting of the set $[2,k+1]$ and all sets that contain $1$ and intersect $[2,k+1]$. Such family has size ${n-1\choose k-1}-{n-k-1\choose k-1}+1$, which is much smaller than ${n-1\choose k-1}$ provided $n$ is large as compared to $k$.\\

This paper is mostly concerned with the properties of large intersecting families. It is separated in several relatively independent parts covering different topics. Below we give the summary of the topics covered and the structure of the paper.
\setlist{leftmargin=2cm}
\begin{enumerate}
  \item[Section \ref{sec3}:] In this section we give a conclusive version of the Frankl's degree theorem (see Theorems~\ref{thm1},~\ref{thmfr} in the next section). The main result gives the precise dependence of size of the family on (lower bound on) the number of sets not containing the most popular element. The result then is extended to cover the equality case, as well as the weighted case and the case of cross-intersecting families. It, in particular, strengthens the results of \cite{Fra1}, \cite{FT}, \cite{KZ}.
  \item[Section \ref{sec4}:] In this section we address the question, what is the structure of large intersecting families. Unlike in Section~\ref{sec3}, where, due to the Kruskal-Katona theorem, we have to deal with lexicographical families only (for definitions and necessary results, see the next section), in this section we work with general families. We prove several results, which, in particular, cover a large part of results of \cite{HK}, \cite{KostM}, and extend them much farther. One important advantage, as compared to the results of \cite{KostM}, is that the results do not require $n$ to be large w.r.t. $k$.
  \item[Section \ref{sec5}:] In this section we obtain degree and subset degree versions of the Erd\H os-Ko-Rado theorem and the Hilton-Milner theorem. This answers some of the questions posed by Huang and extends the results of \cite{HZ}, \cite{FT6}, \cite{FHHZ}.
  \item[Section \ref{sec6}:] In this section we obtain an upper bound on the minimum degree of a family $\ff\subset {[n]\choose k}$, which does not contain $s$ pairwise disjoint sets. The results extends the analogous result of \cite{HZ} for large $k$. The problem in question may be seen as a  degree vesion of the Erd\H os Matching Conjecture.
\end{enumerate}


Some of the results in Section~\ref{sec4} depend on the results from Section~\ref{sec5}, while Section~\ref{sec6} is relatively independent from the previous two sections. Section~\ref{sec6} is completely independent.

In the next section we give some of the definitions and results crucial for the paper. In Section~\ref{sec7} we give some concluding remarks and state several open problems.

\section{Preliminaries}\label{sec2}

The \underline{degree} $\delta(i)$ of an element $i$ is the number of sets from the family containing it.
For a family $\ff$ the \underline{diversity} $\gamma(\ff)$ is the quantity $|\ff|-\Delta(\ff)$, where $\Delta(\ff)$ is the \underline{maximal degree} of an element in $\ff$. Below  we discuss the connection between the diversity and the size of an intersecting family.

We will need the following result due to Kupavskii and Zakharov \cite{KZ}. It is a slightly stronger version of Frankl's result \cite{Fra1}:

\begin{thm}[\cite{KZ}]\label{thm1} Let $n>2k>0$ and $\ff\subset {[n]\choose k}$ be an intersecting family. Then, if $\gamma(\ff)\ge {n-u-1\choose n-k-1}$ for some real $3\le u\le k$, then \begin{equation}\label{eq01}|\ff|\le {n-1\choose k-1}+{n-u-1\choose n-k-1}-{n-u-1\choose k-1}.\end{equation}
\end{thm}

The bound from Theorem \ref{thm1} is sharp for integer $u$, as witnessed by the example
$$\mathcal H_u:=\{A\in {[n]\choose k}:[2,u]\subset A\}\cup\{A\in{[n]\choose k}: 1\in A, [2,u]\cap A\ne \emptyset\}.$$
 We note that the case $u = k$ of Theorem~\ref{thm1} is precisely the Hilton-Milner theorem.

To allow the reader to compare Theorem~\ref{thm1} and the original Frankl's theorem, let us state it here.
\begin{thm}[\cite{Fra1}]\label{thmfr} Let $n>2k>0$ and $\ff\subset {[n]\choose k}$ be an intersecting family. Then, if $\Delta(\ff)\le {n-1\choose k-1}-{n-u-1\choose k-1}$ for some integer $3\le u\le k$, then \begin{equation*}|\ff|\le {n-1\choose k-1}+{n-u-1\choose n-k-1}-{n-u-1\choose k-1}.\end{equation*}\end{thm}

We note that Theorem~\ref{thmfr} is immediately implied by Theorem~\ref{thm1}, while in the other direction it is not true, even for the integer values of $u$. \\

One of the key ingredients in the proofs of several theorems is the Kruskal-Katona theorem. Below we state it in terms of lexicographical orderings. Let us first give some definitions.
A \underline{lexicographical order} (lex) $<$ on the sets from ${[n]\choose k}$ is an order, in which $A<B$ iff  the minimal element of $A\setminus B$ is smaller than the minimal element of $B\setminus A$.
 For $0\le m\le {n\choose k}$ let $\mathcal L(m,k)$ be the collection of $m$ largest sets with respect to lex.

 We say that two families $\mathcal A,\mathcal B$ are \underline{cross-intersecting}, if for any $A\in\mathcal A, B\in \mathcal B$ we have $A\cap B\ne \emptyset$.

\begin{thm}[\cite{Kr},\cite{Ka}]\label{thmHil}Suppose that $\mathcal A\subset {[n]\choose a}, \mathcal B\subset {[n]\choose b}$ are cross-intersecting. Then the families $\mathcal L(|\mathcal A|,a),\mathcal L(|\mathcal B|, b)$ are also cross-intersecting.
\end{thm}

\section{The complete diversity version of Frankl's theorem}\label{sec3}

In this section we study in detail the relationship between the diversity of an intersecting family and its size. First, note that, if the value of diversity is given precisely, then it is very easy to determine the largest intersecting family with such diversity using Theorem~\ref{thmHil}. Studying the size of an intersecting family with given upper bounds on diversity is not interesting: in general, the smaller the diversity is, the larger families with such diversity exist.

In this section we obtain a strengthened version of Theorem~\ref{thm1}, which will tell {\it exactly}, how large an intersecting family may be, given a {\it lower} bound on its diversity. We give all ``extremal'' values of diversity and the sizes of the corresponding families.


The difficulty to obtain such a version of Theorem~\ref{thm1} is that, while Theorem~\ref{thmHil} gives a very strong and clear characterisation of families with fixed diversity, the size of the family is not monotone w.r.t. diversity (the size of the largest family with a given diversity does not necessarily decrease when diversity increases, although it is true in most cases), and an extra effort is needed to find the right point of view.\\

We will give two versions of the main theorem of this section. First, we give a ``numerical'' version with explicit sharp bounds on the size of an intersecting family. It may be more practical to apply in some cases, but it is difficult to grasp, what is hidden behind the binomial coefficients in the formulation. Thus, later in the section (and as an intermediate step of the proof), we will give a ``conceptual'' version of our main theorem. We note that the proof that we present is completely computation-free.  In the later parts of this section we present strengthenings and generalisations of our main result. We will settle the equality case in Theorems~\ref{thmfull1},~\ref{thmfull2}, as well as consider the weighted case and a generalization to the case of general cross-intersecting families.

We note that the main results of the section are meaningful for any $k\ge 3$.\\

The following representation of natural numbers is important for the (classic form of) the Kruskal-Katona theorem. Given positive integers $\gamma$ and $k$, one can always write down $\gamma$ {\it uniquely} in the \underline{$k$-cascade form}:
$$\gamma = {a_k\choose k}+{a_{k-1}\choose k-1}+\ldots +{a_s\choose s}, \ \ a_k>a_{k-1}>\ldots >a_s\ge 1.$$

For the sake of comparison, let us state the classical version of the Kruskal-Katona theorem (equivalent to Theorem~\ref{thmHil}).

\begin{thm}[\cite{Kr},\cite{Ka}] Let $\ff\subset {[n]\choose k}$ and put $$\partial(\ff):=\big\{F'\in {[n]\choose k-1}: F'\subset F \text{ for some }F\in \ff\big\}.$$ If $|\ff| = {a_k\choose k}+\ldots +{a_s\choose s}$, then
$$|\partial(\ff)|\ge {a_k\choose k-1}+{a_{k-1}\choose k-1}+\ldots +{a_s\choose s-1}.$$
\end{thm}


  To state our theorem, we need to represent the diversity of an intersecting family in a cascade form. Given a number $\gamma$, let us write it in the  $(n-k-1)$-cascade form in the following particular way:
  $$\gamma = {n-b_1\choose n-k-1}+{n-b_{2}\choose n-k-2}+\ldots+{n-b_{s_b}\choose n-s-1},$$
  where $1\le b_1<b_2<\ldots<b_{s_b}.$
   Define $T_{\gamma}:=\{b_1,\ldots,b_{s_b}\}$ and put $S_{\gamma}:=\{b_{s_b}\}\cup ([2,b_{s_b}]\setminus T_{\gamma})$. Note that $S_{\gamma}\cup T_{\gamma} = [2,b_{s_b}]$ and $S_{\gamma}\cap T_{\gamma} = \{b_{s_b}\}$. We assume that $S_{\gamma} = \{a_1,\ldots, a_{s_a}\}$, where $2\le a_1<\ldots<a_{s_a} = b_{s_b}$. \\

   We call a number $\gamma$ \underline{resistant}, if the following holds:
   \begin{enumerate}
   \item $s_a = |S_{\gamma}|\le k$, $s_b = |T_{\gamma}|\le k-1$;
   \item $b_i>2i+2$ for each $i\in [s_b]$;
    \item For convenience, we assume that ${n-4\choose k-3}$ is a resistant number.
   \end{enumerate}

Thus, in particular, any integer $\gamma>{n-4\choose k-3}$ will have ${n-4\choose k-3}$ as one of the members in the $(n-k-1)$-cascade form, so such $\gamma$ is not resistant.

Let $1 = \gamma_1<\gamma_2<\ldots <\gamma_{m} = {n-4\choose k-3}$ be all the resistant numbers. Put $\gamma_0=0$  for convenience.


\begin{thm}\label{thmfull1}  Let $n>2k\ge6$. Consider an intersecting family $\ff\subset {[n]\choose k}$. Suppose that $\gamma_{l-1}<\gamma (\ff)\le\gamma_l$ for $l\in [m]$ and that the representation of $\gamma_l$ in the $(n-k-1)$-cascade form is
$$\gamma_l = {n-b_1\choose n-k-1}+{n-b_{2}\choose n-k-2}+\ldots +{n-b_{s_b}\choose n-s-1},$$
then \begin{equation}\label{eqfull1}|\ff|\le {n-a_1\choose n-k}+{n-a_{2}\choose n-k-1}+\ldots +{n-a_{s_a}\choose n-s_a}+\gamma_l,\end{equation}
where $\{b_1,\ldots, b_{s_b}\} = T_{\gamma}$ and $\{a_1,\ldots, a_{s_a}\} = S_{\gamma}$.

 The expression in the right hand side of \eqref{eqfull1} strictly decreases as $l$ increases.

Moreover, the presented bound is sharp: for each $l =1,\ldots, m$ there exists an intersecting family with diversity $\gamma_l$ which achieves the bound in \eqref{eqfull1}.
\end{thm}

\vskip+0.2cm
Let us first try to familiarize the reader with the statement of the theorem. We have $\gamma_i = i$ for $i = 1,\ldots, k-3$. Indeed,  for $1\le \gamma <n-k-1$ we have $\gamma  = {n-k-1\choose n-k-1}+{n-k-2\choose n-k-2}+\ldots +{n-k-\gamma\choose n-k-\gamma}$. Thus, for any such $\gamma$ we have $T_{\gamma} = [k+1,k+\gamma]$ and $S_{\gamma} = [2,k]\cup \{\gamma\}$. Thus, the first condition of the definition of a resistant number is satisfied if $\gamma\le k-1$. However, the second condition is only satisfied when $k+\gamma>2\gamma+2$, which is equivalent to $\gamma\le k-3$ (or when $k=3$, $\gamma=1$). From the discussion above we also get that for $k>3$ $\gamma_{k-2} = {n-k\choose n-k-1}=n-k.$

\begin{prop} The bound in Theorem~\ref{thmfull1} is always at least as strong as the bound in Theorem~\ref{thm1} for intersecting $\ff\subset {[n]\choose k}$ with diversity $\gamma(\ff)\le {n-4\choose k-3}$.
\end{prop}
\begin{proof} Let us first compare the statement of Theorem~\ref{thmfull1} with the statement of Theorem~\ref{thm1} for $\gamma_l := {n-u-1\choose n-k-1}$ with integer $u$. Such $\gamma_l$ is resistant for any $u\in [3,k]$ (note that ${n-4\choose k-3}$ is resistant due to the exceptional condition 3 in the definition). Moreover, it is not difficult to see that, if we substitute such $\gamma_l$ in \eqref{eqfull1}, then we will get the bound $$|\ff|\le {n-2\choose n-k}+\ldots +{n-u-1\choose n-k-u+1}+\gamma_l = {n-1\choose n-k}-{n-u-1\choose n-k-u}+{n-u-1\choose n-k-1},$$ which is exactly the bound \eqref{eq01}. However, we are getting it here in more relaxed assumptions: while we know that this bound is sharp for $\gamma(\ff) = \gamma_l$, Theorem~\ref{thmfull1} also tells us that for $\gamma(\ff)>\gamma_l$ the bound would be strictly worse (and provides us with a possibility to extract, how much worse). Moreover, even if $\gamma_{l-1}<\gamma(\ff)\le \gamma_l$, we are still getting the same upper bound.

Moving to the proof of the proposition, the function in the right hand side of \eqref{eq01} monotone decreasing as $\gamma(\ff)$ increases (and thus $u$ decreases). Therefore, to show that the bound \eqref{eqfull1} is stronger than \eqref{eq01}, it is sufficient to show it for values $\gamma_l$, $l=0,\ldots, m$. But for each of these values the bound \eqref{eqfull1} is sharp, so \eqref{eq01} can be only weaker than \eqref{eqfull1} for these values.

Thus, clearly, Theorem~\ref{thmfull1} is a strengthening of Theorem~\ref{thm1}.
\end{proof}

As a matter of fact, we can replace the bound in \eqref{eq01} with {\it any} monotone decreasing function of $\gamma(\ff)$, provided that the bound holds for each $\gamma_l$. \\

Finally, let us mention that we state Theorem~\ref{thmfull1} only for $\gamma(\ff)\le {n-4\choose k-3}$, since Theorem~\ref{thm1} already gives us the bound $|\ff|\le {n-2\choose k-2}+2{n-3\choose k-2}$ if $\gamma(\ff)\ge {n-4\choose k-3}$, and we cannot get any better bound in this range. However, we are going to analyze the cases when the equality in the inequality for $|\ff|$ above can be attained. It is worth mentioning that the intersecting family $\{F\in {[n]\choose k}: |F\cap [3]|\ge 2\}$ attains the bound above on the cardinality, and has diversity ${n-3\choose k-2}$. We also note that in a recent work \cite{Kup21} the author managed to prove  that for $n>ck$ with some constant independent of $k$ there are no intersecting families $\ff\subset{[n]\choose k}$ with diversity bigger than ${n-3\choose k-2}$.\\

Our next goal is to state the ``conceptual'' version of Theorem~\ref{thmfull1}.
We will need certain preparations. We will use the framework and some of the ideas from \cite{FK1}, as well as from \cite{KZ}. First of all, we pass to the cross-intersecting setting in a standard way: given an intersecting family $\ff$, we consider two families
\begin{align*}
\ff(1):=&\{F\setminus\{1\}:1\in F\in\ff\}\ \ \ \ \ \text{and}\\
\ff(\bar 1):=&\{F:1\notin F\in\ff\},
\end{align*}

and, applying Theorem~\ref{thmHil}, from now on and until the end of the section assume that $\ff(1)=\mathcal L(|\ff(1)|,k-1), \ff(\bar 1) = \mathcal L(|\ff(\bar 1)|, k)$. Note that $\ff(1),\ff(\bar 1)\subset 2^{[2,n]}$. For shorthand we denote $\aaa:=\ff(1), \bb:=\ff(\bar 1)$. In the remaining part dedicated to Theorem~\ref{thmfull1} we will be mostly working with the set $[2,n]$, in order not to confuse the reader and to keep clear the relationship between the diversity of intersecting families and the sizes of pairs of cross-intersecting families.\\

Both $\aaa$ and $\bb$ are determined by their lexicographically last set. In this section we use lexicographical order on $2^{[2,n]}$, which is defined as follows: $A<B$ iff $A\supset B$ or the minimal element of $A\setminus B$ is smaller than the minimal element of $B\setminus A$. Let us recall some notions and results from \cite{FK1} related to the Kruskal-Katona theorem and cross-intersecting families. For a set $S\subset [n]$, $1\in S$ and $|S\cap [2,n]|\le a$, we define
$$\mathcal L(S,a):=\big\{A\in {[2,n]\choose a}:A<S\cap[2,n]\big\}.$$
For example, the family $\{G\in{[2,n]\choose 10}: 2\in G, G\cap\{3,4\}\ne \emptyset\}$ is the same as the family $\mathcal L(S,10)$ for $S=\{2,4\}$. If $\G=\mathcal L(S,a)$ for a certain set $S$, then we say that $S$ is the {\it characteristic set} of $\G$.
Note that, for convenience, we assume that $1\in S$ (motivated by the fact that $S$ is the characteristic set for the subfamily of all sets containing 1 in the original family), while $T\subset [2,n]$.

We say that two sets $S\subset [n]$ and $T\subset [2,n]$  \underline{form a resistant pair}, if the following holds. Assuming that the largest element of $T$ is $j$, we have
\begin{enumerate}
  \item $S\cap T = \{j\},\ S\cup T = [j]$, $|S|\le k,$ $|T|\le k$;
  \item for each $4\le i\le j$ we have $|[i]\cap S|< |[i]\setminus S|$ (this condition, roughly speaking, says that in each $[i]$ there are more elements in $T$ than in $S$);
  \item For convenience, we include the pair $T =\{2,3,4\}, S = \{1,4\}$ in the list of resistant pairs.
\end{enumerate}
Note that 2 implies that $T\supset \{2,3,4\}$ for each resistant pair. There is a close relationship between this notion and the notion of a resistant number, which we discuss a bit later. Let us first give the ``conceptual'' characteristic set version of Theorem~\ref{thmfull1}. For convenience, we put $T_0 = [2,n]$ to correspond to the empty family.

\begin{thm}\label{thmfull2} Let $n>2k\ge6$. Consider all resistant pairs $S_l\subset [n],\ T_l\subset [2,n]$, where $l\in [m]$. Assume that $T_0<T_1<T_2<\ldots <T_m$. Then
\begin{equation}\label{eqfull2} |\mathcal L(S_{l-1},k-1)|+|\mathcal L(T_{l-1},k)|>|\mathcal L(S_{l},k-1)|+|\mathcal L(T_{l},k)| \ \ \ \text{ for each }l \in [m],\end{equation}
and any cross-intersecting pair of families $\mathcal A\subset {[2,n]\choose k-1},\ \mathcal B\subset {[2,n]\choose k}$ with $|\mathcal L(T_{l-1},k)|<|\bb|\le |\mathcal L(T_l,k)|$ satisfies \begin{equation}\label{eqfull3} |\aaa|+|\bb|\le |\mathcal L(S_l,k-1)|+|\mathcal L(T_l, k)|.\end{equation}

In terms of intersecting families, if $\ff\subset {[n]\choose k}$ is intersecting and $|\mathcal L(T_{l-1},k)|<\gamma(\ff)\le |\mathcal L(T_l,k)|$, then $|\ff|\le |\mathcal L(S_l,k-1)|+|\mathcal L(T_l, k)|$.
\end{thm}

First of all, we remark that the intersecting part is clearly equivalent to the second statement of the cross-intersecting part. Second, Proposition~\ref{prop9} below shows that the families $L(S_l,k-1)$ and $L(T_{l},k)$ are cross-intersecting and thus \eqref{eqfull3} is sharp. Now let us deduce Theorem~\ref{thmfull1} from Theorem~\ref{thmfull2}.

\begin{proof}[Reduction of Theorem~\ref{thmfull1} to Theorem~\ref{thmfull2}]
 For any set $T\in [2,n]$ with the largest element $s_b$ we can compute the size of the family $|\mathcal L(T, k-1)$ as follows. Find the first element $b_1\in [2,n]$, which is missing from $T$, and consider the family with characteristic set $T_1:=(T\cap [b_1])\cup \{b_1\}$. The size of this family is ${n-b_1\choose k-b_1+1} = {n-b_1\choose n-k-1}$, since actually $T_1 = [2,b_1]$. Since $T_1<T$, this family is a subfamily of  $\mathcal L(T, k-1)$. At each new step we find the next (not found yet) element $b_i$, which is missing from $T$, the set $T_i:=(T\cap [b_i])\cup b_i$, and count the sets which belong to $\mathcal L(T_i,k-1)\setminus L(T_{i-1},k-1)=\big\{F\in {[2,n]\choose k}:F\cap [b_i]=T_i\big\}$. Their number is precisely ${n-b_i\choose k-|T_i|} = {n-b_i\choose n-k-|[b_i]\setminus T_i|}$. But since we are stopping at every element that is not included in $T$, we get that $|[b_i]\setminus T_i| = |[b_{i-1}\setminus T_{i-1}|+1 = \ldots = i$. Therefore, the last binomial coefficient is ${n-b_i\choose n-k-i}$. At some point $T_i = T$, and we stop the procedure, including the sets $F\in {[2,n]\choose k-1}$ that satisfy $F\cap [s_b] = T$. We get that
$$|\mathcal L(T,k-1)| = {n-b_1\choose n-k-1}+{n-b_2\choose n-k-2}+\ldots +{n-b_{s_b}\choose n-s_b-1},$$
the $(n-k-1)$-cascade form! Moreover, we know that the set $\{b_1,\ldots, b_{s_b}\}$ is exactly the set $T_{\gamma}$ from before the definition of a resistant number, and we have $T_{\gamma} = ([2,s_b]\setminus T)\cup \{s_b\}$ and thus $T = S_{\gamma}$.

Therefore, if $S,T$ is a resistant pair, then, representing $\gamma:=|\mathcal L(T,k-1)|$ in an $(n-k-1)$-cascade form, we get that $T = S_{\gamma}$ and $S\cap [2,n] = T_{\gamma}$. This immediately shows that $T_{\gamma}$ and $S_{\gamma}$ satisfy condition 1 of the definition of a resistant number. The implication in the other direction follows in the same way.  Condition 2 of the definition of a resistant pair is equivalent to the statement that for each $i$ $1+|T_{\gamma}\cap [i]|<|[2,i]\setminus T_{\gamma}|$, which is, in turn, the same as saying $b_i>2i+2$. Finally, it is clear that $\gamma = {n-4\choose k-3}$ correspond to the characteristic set $\{2,3,4\}$.

We conclude that $T_l,S_l$ form a resistant pair if and only if $|\mathcal L(T_l,k-1)|$ is a resistant number. Doing calculations as above, one can conclude that
$$|\mathcal L(S_l,k)| = {n-a_1\choose n-k}+{n-a_{2}\choose n-k-1}+\ldots +{n-a_{s_a}\choose n-s_a}.$$

Given that, it is clear that the inequality \eqref{eqfull2} is equivalent to the statement saying that the right hand side of \eqref{eqfull1} is strictly monotone, and that the \eqref{eqfull1} is equivalent to \eqref{eqfull3}.

Finally, the sharpness claimed in Theorem~\ref{thmfull1} follows right away from the fact that the inequality \eqref{eqfull3} in Theorem~\ref{thmfull2} is expressed in terms of families. That is, the pair $\mathcal L(S_l,k-1)$ and $\mathcal L(T_l,k)$ provides us with such an example.
\end{proof}
\vskip+0.2cm

Before proving Theorem~\ref{thmfull2}, let us first shed some light on pairs of cross-intersecting lexicographic families.
We say that two sets $S$ and $T$ in $[2,n]$ \underline{strongly intersect}, if there exists a positive integer $j$, such that $S\cap T\cap[2,j]=\{j\}$ and $S\cup T\supset [2,j]$.
 The following easy proposition was proven in \cite{FK1}:

\begin{prop}[\cite{FK1}]\label{prop9} Let $A$ and $B$ be subsets of $[2,n]$, $|A|\le a, |B|\le b$, and $|[2,n]|=n-1\ge a+b$. Then $\mathcal L(A,a)$ and $\mathcal L(B,b)$ are cross-intersecting iff $A$ and $B$ strongly intersect.
\end{prop}

We say that $\mathcal A\subset {[2,n]\choose a}$ and $\mathcal B\subset {[2,n]\choose b}$ form a \underline{maximal cross-intersecting pair}, if whenever $\mathcal A'\subset {[2,n]\choose a}$ and $\mathcal B'\subset {[2,n]\choose b}$ are cross-intersecting with $\mathcal A'\supset \mathcal A$ and $\mathcal B'\supset \mathcal B$, then necessarily $\mathcal A = \mathcal A'$ and $\mathcal B = \mathcal B'$ holds.

The following proposition from \cite{FK1} is another important step in our analysis.

\begin{prop}[\cite{FK1}]\label{cross2} Let $a$ and $b$ be positive integers, $a+b\le n-1$. Let $P$ and $Q$ be non-empty subsets of $[2,n]$ with $|P|\le a$, $|Q| \le b$. Suppose that $P$ and $Q$ intersect strongly in their last element. That is, there exists $j$, such that $P\cap Q = \{j\}$ and $P\cup Q = [2,j]$. Then $\mathcal L(P,a)$ and $\mathcal L(Q,b)$ form a maximal pair of cross-intersecting families.

Inversely, if $\mathcal L(m,a)$ and $\mathcal L(r,b)$ form a maximal pair of cross-intersecting families, then it is possible to find sets $P$ and $Q$ such that $\mathcal L(m,a)=\mathcal L(P,a)$, $\mathcal L(r,b)=\mathcal L(Q,b)$ and $P,Q$ satisfy the above condition. \end{prop}

We note that it may be helpful to interpret the strong intersection property, as well as the lexicographical order etc. in terms of $\{0,1\}$-vectors: 1 on the $i$-th position if $i$ is contained in the corresponding set, and $0$ otherwise.\\

Now we pass on to the proof of the cross-intersecting version of Theorem~\ref{thmfull2}. Fix a cross-intersecting pair of families $
\aaa,\bb$ as in the theorem.  There are three important reduction steps, which restrict the class of cross-intersecting families which we need to consider. First, as we have already mentioned, we assume that $\aaa = \mathcal L(S,k-1)$ and $\bb = \mathcal (T,k)$ for some characteristic sets $S, T$.
 Second, in view of Proposition~\ref{cross2}, we may assume that $\aaa = \mathcal L(S,k-1),\ \bb=\mathcal L(T,k)$ for some sets $S,T$ that strongly intersect in their last element.   Note also that $|[2,n]|=n-1\ge k+(k-1)$, so we do not have to worry about this condition in the propositions above.

Recall that we aim to maximize $|\aaa|+|\bb|$ given a lower bound on $|\bb|$. The third reduction step is the following lemma.

\begin{lem}\label{lemdiv} Consider a pair of cross-intersecting families $\aaa\subset {[2,n]\choose k-1},\ \bb\subset {[2,n]\choose k}$. Suppose that $\aaa = \mathcal L(S,k-1),\ \bb=\mathcal L(T,k)$ for some sets $S\subset [n]$, $T\subset [2,n]$ that strongly intersect in their last element $j$.

Assume that $S$ and $T$ do not form a resistant pair, that is, there exists $4\le i\le j$, such that $\big|[i]\setminus S\big|\le \big|S\cap [i]\big| $. Put $T':=[i]\setminus S$ and choose $S'$ so that it strongly intersects with $T'$ in its largest element. Then the families $\aaa'\subset {[2,n]\choose k-1},\ \bb\subset {[2,n]\choose k}$ with characteristic sets $S', T'$ are cross-intersecting and  satisfy $|\aaa'|+|\bb'|\ge |\aaa|+|\bb|$ and $|\bb'|> |\bb|$.

Moreover, if $\big|[i]\setminus S\big|< \big|S\cap [i]\big|$, then we may conclude that $|\aaa'|+|\bb'|> |\aaa|+|\bb|$.
\end{lem}

\begin{proof}[Proof of Lemma~\ref{lemdiv}] First, recall that $1\in S$. Since $S'$ and $T'$ are strongly intersecting, the families $\aaa'$, $\bb'$ are cross-intersecting. Next, clearly, $T'\subsetneq T$ and thus $\bb'\supsetneq \bb$. Therefore, we only have to prove $|\aaa|+|\bb|\le |\aaa'|+|\bb'|$.

Consider the following families:
\begin{align*}
\mathcal P_a:=&\big\{P\in {[2,n]\choose k-1}: P\cap [i] = [2,i]\cap S\big\},\\
\mathcal P_b:=&\big\{P\in {[2,n]\choose k}: P\cap [i] = [i]\setminus S\big\}.
\end{align*}
Most importantly, we have $\aaa\setminus \mathcal P_a = \aaa',$ $\bb\cup \mathcal P_b = \bb'$. Let us show, e.g., the first equality.  We have $S'< S<S\cap [i]$, therefore $\aaa'\subset \aaa\subset \mathcal L(S\cap [i],k-1)$. On the other hand, we claim that $S'$ and $S\cap [i]$ are two consecutive sets in the lexicographic order on $[i]$. Indeed, assume that the largest element of $T'$ is $j'$.
If $j' = i$, then $S'\supset S\cap [i]$, $(S'\setminus S)\cap [i] = \{i\}$, which proves it in this case. If $j'<i$, then $[j'+1,i]\subset S\cap [i]$, $j'\notin S\cap [i]$. It is easy to see that the set that precedes $S\cap [i]$ in the lexicographic order on $[i]$ ``replaces'' $[j'+1,i]$ with $\{j'\}$, that is, it is $S'$. Therefore,
$$\aaa \setminus \aaa' \subset \mathcal L(S\cap [i],k-1)\setminus \aaa' = \mathcal L(S\cap [i],k-1)\setminus \mathcal L(S',k-1) = \mathcal P_a,$$ which, together with the fact that $\mathcal P_a$ and $\aaa'$ are disjoint, is equivalent to the equality we aimed to prove.

Next, consider a bipartite graph $G$ with parts $\mathcal P_a,\mathcal P_b$ and edges connecting disjoint sets. Then, due to the fact that $\aaa$ and $\bb$ are cross-intersecting, $(\aaa\cap \mathcal P_a)\cup (\bb\cap \mathcal P_b)$ is an independent set in $G$.

The graph $G$ is biregular, and therefore the largest independent set in $G$ is one of its parts. We have $|\mathcal P_a| = {n-i\choose s_a}$, $|\mathcal P_b| = {n-i\choose s_b}$, where $s_a = k-|[i]\cap S|,\ s_b = k-|[i]\setminus S|$. By the condition from the lemma, we have $s_b\ge s_a$, and, since $n-i\ge s_a+s_b$, we have $|\mathcal P_b|\ge |\mathcal P_a|$. We conclude that $|\mathcal P_b|$ is the largest independent set in $G$, so $|\mathcal P_b|\ge (\aaa\cap \mathcal P_a)\cup (\bb\cap \mathcal P_b),$ and therefore
$$|\aaa'|+|\bb'|-(|\aaa|+|\bb|) = |\mathcal P_b|-|\aaa\cap \mathcal P_a|-|\bb\cap \mathcal P_b|\ge 0.$$
If  $\big|[i]\setminus S\big|< \big|S\cap [i]\big|$, then $|\mathcal P_b|> |\mathcal P_a|$ and $\mathcal P_b$ is the {\it unique} independent set of maximal size in $G$. Thus, we have strict inequality in the displayed formula above.
\end{proof}

Our next goal is to understand how do the resistant pairs behave. More specifically, we aim to show that \eqref{eqfull2} holds: that sizes of resistant cross-intersecting families are increasing as the size of the second family decreases.


\begin{lem}\label{lemres} Consider a resistant pair of cross-intersecting families $\aaa\subset{[2,n]\choose k-1},\ \bb\subset {[2,n]\choose k}$, with characteristic sets $S,T$, respectively, and another such resistant pair $\aaa'\subset{[2,n]\choose k-1},\ \bb'\subset {[2,n]\choose k}$ with characteristic sets $S',T'$. Assume also that $T\ne T'$.

If $T'<T$, then $|\bb'|< |\bb|$ and $|\aaa|+|\bb|<|\aaa'|+|\bb'|$.
\end{lem}

Roughly speaking, while for general lexicographic pairs of families the size is not monotone w.r.t. the lexicographic order, it is monotone for resistant pairs, which are maximal with respect to the properties that interest us.

Remark that, since $T'<T$, then $T'\ne \{2,3,4\}$ and thus $S',T'$ must satisfy the second condition from the definition of a resistant pair. We also note that we do not use the property that $S,T$ form a resistant pair. The proof also works for $T' = T_0 (=[2,n])$.

The proof of this lemma is based on biregular bipartite graphs and is very similar to the proof of Lemma~\ref{lemdiv}, although is a bit trickier.

\begin{proof} First of all, it is clear that in the conditions of the lemma we have $|\bb'|< |\bb|$. The rest of the proof is concerned with the inequality on the sums of sizes. We will consider two cases depending on how do the sets $T$ and $T'$ relate.\\

\textbf{Case 1. \pmb{$T'\nsupseteq T$.} }  Find the smallest $i$, such that one of the sets contain $i$, while the other does not. Since $T'<T$, we clearly have $i\in T',\ i\notin T$. Consider the set $T'' = T'\cap [i]$. Then we clearly have $T'<T''<T$ and $T''\subset T'$. Accordingly, put $S''$ to be $\{i\}\cup ([i]\setminus T'')$, and consider the cross-intersecting families $\aaa''\subset {[2,n]\choose k-1},\ \bb''\subset {[2,n]\choose k}$, which have characteristic vectors $S''$ and $T''$, respectively. First, note that $\aaa''$ and $\bb''$ is a resistant pair: it follows from the definition of $T''$. We claim that $|\aaa''|+|\bb''|> |\aaa|+|\bb|.$

We prove the inequality above as in Lemma~\ref{lemdiv}, but the roles of $S$ and $T$ are now switched. Consider a bipartite graph $G$ with parts

\begin{align*}
\mathcal P_a:=&\big\{P\in {[2,n]\choose k}: P\cap [i] = [2,i]\setminus T\big\},\\
\mathcal P_b:=&\big\{P\in {[2,n]\choose k}: P\cap [i] = [i]\cap T\big\},
\end{align*}
and edges connecting disjoint sets.
Similarly, we have $\aaa\cup \mathcal P_a = \aaa'',$ $\bb\setminus \mathcal P_b = \bb''$. Indeed, let us verify, e.g., the second equality. All sets $P$ such that $P\cap [i] < T''$ are in $\bb$ and in $\bb\setminus \mathcal P_b$, as well as in $\bb''$, since $T''\lneq T\cap [i]$. On the other hand, if we restrict to $[i]$, the sets $T\cap [i]$ and $T''$ are consecutive in the lexicographic order, and so any set $B$ from $\bb$ such that $B>T''$ must satisfy $B\cap [i] = T\cap [i]$. Therefore, $\bb\setminus \bb''\subset \mathcal P_b$ and $\bb\setminus \mathcal P_b = \mathcal B''$.

Since $\aaa$ and $\bb$ are cross-intersecting, the set $(\aaa\cap \mathcal P_a)\cup (\bb\cap \mathcal P_b)$ is independent in $G$. On the other hand, the largest independent set in $G$ has size $\max\{|\mathcal P_a|, |\mathcal P_b|\}$. Since the pair $\aaa,$ $\bb$ is resistant (and that $i$ is not the last element of $T$), we have that $|[i]\cap T|=|[i]\setminus S|>|[i]\cap S|=|[i]\setminus T|$, which implies $|\mathcal P_a|={n-i\choose |[2,i]\setminus T|}> {n-i\choose |[i]\cap T|} = |\mathcal P_b|$, and thus $\mathcal P_a$ is the (unique) maximal independent set in $G$. We have

$$|\aaa''|+|\bb''|-(|\aaa|+|\bb|) = |\mathcal P_a|-|\aaa\cap \mathcal P_a|-|\bb\cap \mathcal P_b|>0,$$
and the desired inequality is proven. Therefore, when comparing $T'$ and $T$, we may replace $T$ with $T''$, or rather assume that $T\subset T'$. We have reduced Case 1 to the following case.\\

\textbf{Case 2. \pmb{$T'\supseteq T$.} } Arguing inductively, we may assume that $|T'\setminus T|=1$, and that $T'\setminus T=\{i\}$ for some $i\in [2,n]$. Note that in that case $i$ is the last element of $T', S'$, and that $T'\cap S'=\{i\}$. Consider a bipartite graph $G$ with parts
\begin{align*}
\mathcal P_a:=&\big\{P\in {[2,n]\choose k-1}: P\cap [i] = S'\cap [2,i]\big\},\\
\mathcal P_b:=&\big\{P\in {[2,n]\choose k}: P\cap [i] = [i]\setminus S'\big\},
\end{align*}
and edges connecting disjoint sets. As before, we have $\aaa\cup \mathcal P_a = \aaa',$ $\bb\setminus \mathcal P_b = \bb'$. Using the fact that $|[i]\setminus S'|>|[i]\cap S'|$, we again conclude that $|\aaa'|+|\bb'|>|\aaa|+|\bb|$.
\end{proof}

Now let us put the things together.

\begin{proof}[Proof of Theorem~\ref{thmfull2}] First of all, \eqref{eqfull2} follows from Lemma~\ref{lemres}. Next, given a pair of cross-intersecting families $\aaa,\bb$, we may assume using Proposition~\ref{cross2} that their characteristic sets $S,T$ strongly intersect in their last coordinate. Using Lemma~\ref{lemdiv}, we may further replace them with a resistant pair $\aaa',\ \bb'$ with characteristic vectors $S',\ T'$, such that $T<T'$. Therefore, if $T_{l-1}<T$ and $T_{l-1}\neq T$, then $T_l<T'$, and therefore the pair $\mathcal L(S_l,k-1),\ \mathcal L(T_l,k)$ has at least as big sum of cardinalities as $\aaa'$ and $\bb'$. This completes the proof of the theorem.

We only have to add that, although $T_m = \{2,3,4\}, S_m = \{1,4\}$ do not satisfy the second requirement of the definition of a resistant pair, it does not pose any problems. We still may apply Lemma~\ref{lemres} to this pair. Moreover, if initially the characteristic set $T$ of the family $\aaa$ satisfies $T_{m-1}<T<T_m$, then, using Lemma~\ref{lemdiv} it would be eventually reduced to $T_m$, and thus we may apply it as in other cases.
\end{proof}

Let us discuss some possible strengthenings and generalizations of Theorems~\ref{thmfull1} and~\ref{thmfull2}. First of all, let us determine, for which $T$, $T_{l-1}<T<T_l$, it is possible to have equality in \eqref{eqfull3}, given that $\mathcal A, \bb$ are determined by a strongly intersecting pair of sets $S,\ T$, $|S|, |T|\le k$, which intersect in their last element. We say that a pair of strongly intersecting sets $S,\ T$ as above is \underline{$T_l$-neutral}, if $T$ is obtained in the following recursive way:
\begin{enumerate}
\item $T_l$ is $T_l$-neutral;
\item If $T'$ is $T_l$-neutral, then the set $T:=T'\cup \{x\}$ is $T_l$-neutral, where $x = 2|T'|$.
\end{enumerate}
In practice, this means that, starting from a set $T_l$, we add the element $2|T_l|,$ and then continue adding every other element.

We remark that it is not difficult to see that any $T_l$-neutral pair $S, T$ actually satisfies $T_{l-1}<T\le T_l$. Let us also note that
$T\ne T'$: indeed, from the definition of a resistant pair, the largest element in $T_l$ is at most $2|T_l|$ (actually, it is at most $2|T_l|-3$ for all $l<m$ and equal to $2|T_l|-2$ in the case $l=m$), and every newly added element (via part 2 of the recursive definition) is bigger by 2 than the previously added element.

\begin{thm}\label{thmfulleq} Let $n>2k\ge6$. Consider a pair $\aaa \subset {[2,n]\choose k-1},\ \bb\subset{[2,n]\choose k}$ defined by strongly intersecting sets $S, T$ that intersect in their largest element. If $T_{l-1}<T \le T_l$ for some $l=0,\ldots, m$, then equality in \eqref{eqfull3} holds if and only if the pair $S, T$ is $T_l$-neutral.
\end{thm}

\begin{proof} First, let us show that any $T_l$-neutral pair would have equality in ~\eqref{eqfull3}. We do it inductively. It is clear for a pair involving $T_l$ itself. Assuming it holds for $T'$, let us prove that it holds for $T:=T'\cup \{2|T'|\}$.

Consider the pairs of cross-intersecting families $\aaa, \bb$, $\aaa', \bb'$, corresponding to the $T_l$-neutral pairs of sets $S, T$ and $S', T'$, respectively. Looking in the proof of Lemma~\ref{lemdiv}, given that $|\mathcal P_a|\le |\mathcal P_b|$, the equality in $|\aaa|+|\bb|\le |\aaa'|+|\bb'|$ occur if and only if, first $|\mathcal P_a| = |\mathcal P_b|$, and, second, $\aaa\setminus \aaa' = \mathcal P_a,$ $\bb\setminus \bb' = \mathcal P_b.$ Indeed, in a connected biregular bipartite graph there are only two possible independent sets of maximal size: its parts.

Consider the set $T = T'\cup \{x\}$ and the corresponding set $S$. By the definition of a neutral set, we have $$|[x]\cap S| = |[x]\setminus S|.$$
Therefore, applying the argument of Lemma~\ref{lemdiv} with
\begin{align*}
\mathcal P_a:=&\big\{P\in {[2,n]\choose k-1}: P\cap [x] = [2,x]\cap S\big\},\\
\mathcal P_b:=&\big\{P\in {[2,n]\choose k}: P\cap [x] = [x]\setminus S\big\},
\end{align*}
We get that $k-1-|[2,x]\cap S| = k-|[x]\setminus S|$, which implies $|\mathcal P_a| = |\mathcal P_b|$. Moreover, $\aaa\setminus \aaa' = \mathcal P_a,$ $\bb\setminus \bb' = \mathcal P_b,$ since the sets $S'$ and $S$ are consecutive in the lexicographical order on $[x]$, and the same for $T, T'$. Therefore, $|\aaa'|+|\bb'| = |\aaa|+|\bb|$.\\


In the other direction, take a set $T$, $T_{l-1}<T\le T_l$, and its pair $S$. Consider the corresponding pair of cross-intersecting families $\aaa, \bb$. Then it is easy to see that $T\supset T_l$. (Otherwise, either $T>T_l$, or $T$ must contain and thus precede some other resistant set, which precedes $T_l$, and this would contradict the position of $T$ in the ordering.) Assuming that $x$ is the last element in $T$, we must have
$$|[i]\cap S|<|[i]\setminus S| \ \ \ \ \ \ \ \text{for }4\le i\le x-1.$$
Indeed, otherwise, considering the bipartite graph $G$ with parts $\mathcal P_a$, $\mathcal P_b$ as displayed above for that $i$, we would get that $|\mathcal P_a|\le |\mathcal P_b|$ and both $\mathcal P_a\cap \aaa$ and $\mathcal P_b\cap \bb$ are non-empty. In this case $|\mathcal P_a\cap \aaa|+|\mathcal P_b\cap \bb|<|\mathcal P_a|$, which means that the pair $\aaa', \bb'$ defined by the characteristic sets $T':=T\cap [i]$ and its pair $S'$ would satisfy $|\aaa'|+|\bb'|>|\aaa|+|\bb|$. Moreover, $T'\supset T_l$, so $T<T'\le T_l$ and  $|\aaa'|+|\bb'|\le |\mathcal L(S_l,k-1)|+|\mathcal L(T_l,k)|$. This would contradict the equality in \eqref{eqfull3}.

Therefore, since $S, T$ are not resistant, we have
$$|[x]\cap S|=|[x]\setminus S|.$$
(We cannot have ``$>$'', since otherwise we would have ``$\ge$'' for $i=x-1$.) Removing $x$, we get a set $T'$, and conclude that $x = 2|T'|$. By induction on the size of the set $T$, we may assume that $T'$ is $T_l$-neutral. But then $T$ is $T_l$-neutral as well.
\end{proof}

A slight modification of this argument (with an extension of the definition of a neutral pair to the ones that start with $T_{m+1}:=\{2,3\}$) gives the following:

\begin{prop}\label{propfulleq} Let $n>2k\ge 6$. For any intersecting $\ff\subset {[n]\choose k}$ with ${n-4\choose k-3}<\gamma(\ff)<{n-3\choose k-2}$ we have $|\ff| < {n-2\choose k-2}+2{n-3\choose k-2}$.
\end{prop}

Recall that there are intersecting families $\ff\subset {[n]\choose k}$ for both $\gamma(\ff) = {n-4\choose k-3}$ and ${n-3\choose k-2}$ that have size ${n-2\choose k-2}+2{n-3\choose k-2}$.

 Next, our techniques allow us to give a weighted version of Theorems~\ref{thmfull1} and~\ref{thmfull2}. Assume that, instead of maximising the expression $|\ff|=\Delta(\ff)+\gamma(\ff)$ with a given lower bound on $\gamma(\ff)$, we are maximising the expression $\Delta(\ff)+c\gamma(\ff)$ with some $c>1$. (In terms of cross-intersecting families, we are maximising the expression $|\aaa|+c|\bb|$.) Then the following
is true.

\begin{thm}\label{thmfullw} Let $n>2k\ge 6$. Consider all resistant pairs $S_l\subset [n],\ T_l\subset [2,n]$, where $l=0,\ldots,m$. Assume that $T_0<T_1<T_2<\ldots <T_m$. Put $C:= \frac{n-k-3}{k-3}>1$.  Then
\begin{equation}\label{eqfull4} |\mathcal L(S_{l-1},k-1)|+C|\mathcal L(T_{l-1},k)|\ge|\mathcal L(S_{l},k-1)|+C|\mathcal L(T_{l},k)| \ \ \ \text{ for each }l \in [m],\end{equation}
and any cross-intersecting pair of families $\mathcal A\subset {[2,n]\choose k-1},\ \mathcal B\subset {[2,n]\choose k}$ with $|\mathcal L(T_{l-1},k)|<|\bb|\le |\mathcal L(T_l,k)|$ satisfies $|\aaa|+C|\bb|\le |\mathcal L(S_l,k-1)|+C|\mathcal L(T_l, k)|$.\\

In terms of intersecting families, if $\ff\subset {[n]\choose k}$ is intersecting and $|\mathcal L(T_{l-1},k)|<\gamma(\ff)\le |\mathcal L(T_l,k)|$, then $\Delta(\ff)+C\gamma(\ff)\le |\mathcal L(S_l,k-1)|+C|\mathcal L(T_l, k)|$.
\end{thm}

\begin{proof} The proof of this theorem follows the same steps as the proof of Theorem~\ref{thmfull2}. We sketch the proof of the cross-intersecting version of the theorem.
Using Lemma~\ref{lemdiv}, we may assume that $\aaa$ and $\bb$ form a resistant pair (indeed, otherwise, replacing $\aaa,\ \bb$ with $\aaa'|,$ $|\bb'$ which satisfy $|\aaa'|+|\bb'|\ge |\aaa|+|\bb|$, $|\bb'|>|\bb|$ definitely increases the value of $|\aaa|+C|\bb|$). Then, looking at the proof of Lemma~\ref{lemres}, we see that in each of the cases the bipartite graph $G$ had parts of sizes ${n-i\choose s_a}$, ${n-i\choose s_b}$, where $s_a+s_b = 2k-i$ and $s_a>s_b$. We also know that $s_b\le k-3$. Therefore, even if we put weight ${n-i\choose s_a}/{n-i\choose s_b}$ on each vertex of the part $\mathcal P_b$ and weight 1 on each vertex of the part $\mathcal P_a$, we can still conclude that the independent set of the largest weight in $G$ is $\mathcal P_a$, and the rest of the argument works out as before.
We have
\begin{equation}\frac{{n-i\choose s_a}}{{n-i\choose s_b}}\ge\frac{(n-i-s_b)}{s_a} = \frac{(n-2k+s_a)}{s_a}\ge \frac{(n-k-3)}{k-3}.\end{equation}
Therefore, substituting $\frac{n-k-3}{k-3}$ as the weight in the $\mathcal P_b$ part will work.
\end{proof}



\begin{cor}\label{corweight}
 Let $n>2k\le 6$. For any intersecting family $\ff\subset{[n]\choose k},$ $\gamma(\ff)\le {n-4\choose k-3}$, we have $|\Delta(\ff)|+\frac{n-k-3}{k-3}\gamma(\ff)\le {n-1\choose k-1}$.

 If, additionally, $\ff$ is non-trivial, then $|\Delta(\ff)|+\frac{n-k-3}{k-3}\gamma(\ff)\le {n-1\choose k-1}-{n-k-1\choose k-1}+\frac{n-k-3}{k-3}$.
\end{cor}
\vskip+0.2cm

Is not difficult extend the considerations of this section to the case of cross-intersecting families $\aaa\subset {[n]\choose a}, \bb\subset {[n]\choose b}$, $a<b$. The wording of Theorem~\ref{thmfull2} would stay practically the same. One just need to adjust the definition of a resistant pair. Note that, unlike before in this section, here we will work with cross-intersecting families on $[n]$ (and not $[2,n]$).

We say that two sets $S, T\subset [n]$  \underline{form an $(a,b)$-resistant pair}, if the following holds. Assuming that the largest element of $T$ is $j$, we have
\begin{enumerate}
  \item $S\cap T = \{j\},\ S\cup T = [j]$, $|S|\le a,$ $|T|\le b$;
  \item for each $b-a+2\le i\le j$ we have $|[i]\cap S|-a< |[i]\setminus S|-b$;
  \item The pair $T =\{1,\ldots,b-a+2\}, S = \{1,b-a+2\}$ is resistant.
\end{enumerate}
 Again, for convenience, we put $T_0 = [2,n]$ to correspond to the empty family and $T_{m+1}=\{1,b-a+1\}$ to be an analogue of the set $\{2,3\}$ in this case, where $m$ is the number of resistant pairs.
Here is the theorem, which is an analogue of Theorems~\ref{thmfull2},~\ref{thmfulleq},~\ref{thmfullw} and Proposition~\ref{propfulleq} in the case of general cross-intersecting families. Its proof is a straightforward generalization of the proofs of the respective theorems, thus we omit it.

\begin{thm}\label{thmfullcri} Let $a,b>0$, $n>a+b$. Consider all $(a,b)$-resistant pairs $S_l\subset [n],\ T_l\subset [2,n]$, where $l\in [m]$. Assume that $T_0<T_1<T_2<\ldots <T_m<T_{m+1}$.

1. Then
\begin{equation}\label{eqfull5} |\mathcal L(S_{l-1},a)|+|\mathcal L(T_{l-1},b)|>|\mathcal L(S_{l},a)|+|\mathcal L(T_{l},b)| \ \ \ \text{ for each }l \in [m],\end{equation}
and any cross-intersecting pair of families $\mathcal A\subset {[n]\choose a},\ \mathcal B\subset {[n]\choose b}$ with $|\mathcal L(T_{l-1},b)|<|\bb|\le |\mathcal L(T_l,b)|$ satisfies \begin{equation}\label{eqfull6} |\aaa|+|\bb|\le |\mathcal L(S_l,a)|+|\mathcal L(T_l, b)|.\end{equation}
With an obvious generalization of the notion of a neutral pair, if the families $\mathcal L(|\aaa|,a)$, $\mathcal L(|\bb|,b)$ have characteristic sets $S,T$, then we have equality in \eqref{eqfull6} if and only if $S,T$ is a $T_l$-neutral pair.

2. The same conclusion could be made with $|\mathcal B|$, $|\mathcal L(T_{l-1},b)|$ $|\mathcal L(T_l,b)|$ replaced with $C|\mathcal B|$, $C|\mathcal L(T_{l-1},b)|$ $C|\mathcal L(T_l,b)|$, where $C$ is a constant, $C<\frac{n-b-2}{a-2}.$

3. Denote $t:=a-b-1$. For ${n+t-1\choose a-2}<|\bb|<{n+t\choose a-1}$ we have \begin{small}$$|\aaa|+|\bb|<{n\choose a}-{n+t\choose a}+{n+t\choose a-1} = {n\choose a}-{n+t-1\choose a}+{n+t-1\choose a-2}.$$\end{small}
For both $T_{m+1}=[b+1-a]$, $T_m=[b+2-a]$ and the corresponding $S_{m+1},\ S_m$ we have equality in the inequality above for $\mathcal L(S_i,a), \mathcal L(T_i,b)$, $i=m,m+1$ (note that the second family has cardinality ${n+t\choose a-1}$ and ${n+t-1\choose a-2}$, respectively).
\end{thm}

This theorem generalizes and strengthens many results on cross-intersecting families, in particular, the theorem for cross-intersecting families proven in \cite{KZ} and the following theorem due to Frankl and Tokushige \cite{FT}

\begin{thm}[Frankl, Tokushige, \cite{FT}]\label{lemft} Let $n > a+b$, $a\le b$, and suppose that families $\mathcal F\subset{[n]\choose a},\mathcal G\subset{[n]\choose b}$ are cross-intersecting. Suppose that for some real number $\alpha \ge 1$ we have  ${n-\alpha \choose n-a}\le |\mathcal F|\le {n-1\choose n-a}$. Then \begin{equation}\label{eqft}|\mathcal F|+|\mathcal G|\le {n\choose b}+{n-\alpha \choose n-a}-{n-\alpha \choose b}.\end{equation}
\end{thm}

The deduction is similar to the one we made for Theorems~\ref{thm1} from Theorem~\ref{thmfull1}.\\

One easy corollary of (\eqref{eqfull5} and part 3 of) Theorem~\ref{thmfullcri}, which also appeared in \cite{KZ} and other places, is as follows:
\begin{cor}[\cite{KZ}]
  Let $a,b>0$, $n>a+b$. Let $\aaa\subset {[n]\choose a},\ \bb\subset {[n]\choose b}$ be a pair of cross-intersecting families. Then, if $|\bb|\le {n+a-b-1\choose a-1}$, then
\begin{equation}\label{eqcreasy} |\aaa|+|\bb|\le {n\choose a}.\end{equation}
Moreover, the displayed inequality is strict unless $|\bb|=0$.

If ${n-j\choose b-j}\le |\bb|\le {n+a-b-1\choose a-1}$ for $j\in [b-a+1, b]$, then
\begin{equation}\label{eqcreasy2} |\aaa|+|\bb|\le {n\choose a}-{n-j\choose a}+{n-j\choose b-j}.\end{equation}
Moreover, if the left inequality on $\bb$ is strict, then the inequality in the displayed formula above is also strict.
\end{cor}
Note that the values $|\bb| = {n-j\choose b-j}$ correspond to resistant pairs in Theorem~\ref{thmfullcri}.

\section{Beyond Hilton-Milner theorem}\label{sec4}

Several authors aimed to determine {\it precisely}, what are the largest intersecting families in ${[n]\choose k}$ with certain restrictions. One of such questions was studied by Han and Kohayakawa, who determined precisely, what is the largest family of intersecting families that is neither contained in the Erd\H os-Ko-Rado family, nor in the Hilton-Milner family. In our terms, the question is simply as follows: what is the largest intersecting family with $\gamma(\ff)\ge 2$? The proof of Han and Kohayakawa is quite technical and long. Kruskal-Katona-type arguments allow for a very short and simple proof in the case $k\ge 5$. For $i\in[k]$ let us put
$$\mathcal J_i:=[2,k+1]\cup [i+1,k+i]\cup \big\{F\in{[n]\choose k}:1\in F, F\cap [2,k+1]\ne \emptyset, F\cap [i+1,k+i]\ne \emptyset\big\}.$$
We note that $\mathcal J_i\subset{[n]\choose k}$ and that $\mathcal J_i$ are intersecting. Moreover, $\gamma(\mathcal J_i) = 2$ for $i>1$ and $\mathcal J_1$ is the Hilton-Milner family.

\begin{thm}[\cite{HK}]\label{thmhk} Let $n>2k$, $k\ge 4$. Then any intersecting family $\ff$ with $\gamma(\ff)\ge 2$ satisfies
\begin{equation}\label{eqhk}|\ff|\le {n-1\choose k-1}-{n-k-1\choose k-1}-{n-k-2\choose k-2}+2,\end{equation}
moreover, for $k\ge 5$ the equality is attained only on the families isomorphic to $\mathcal J_2$.
\end{thm}
We note that Han and Kohayakawa proved their theorem for $k=3$, and also resolved the uniqueness case for $k=4$. Unfortunately, this cannot be done in a simple way using our methods. A slightly weaker version of the theorem above (without uniqueness) is a consequence of one of the results obtained by Hilton and Milner \cite{HM} (see \cite{HK}).

It is not so difficult to deduce Theorem~\ref{thmhk} from Theorem~\ref{thmHil} directly, but using Theorems~\ref{thmfull1}, \ref{thmfull2} makes it even easier.
\begin{proof} In terms of Theorem~\ref{thmfull1}, we know that $\gamma_i = i$ for $i\in[k-3]$, and $\gamma_{k-2} = n-k$. Substituting $l=2$ for $k\ge 5$ in \eqref{eqfull1}, we get
$$|\ff|\le {n-2\choose n-k}+{n-3\choose n-k-1}+\ldots+{n-k\choose n-2k+2}+{n-k-2\choose n-2k+1}+2 = $$
$${n-1\choose n-k}-{n-k\choose n-2k+1}+{n-k-2\choose n-2k+1}+2 = {n-1\choose n-k}-{n-k-1\choose n-2k}-{n-k-2\choose n-2k}+2,$$
which is exactly the right hand side of \eqref{eqhk}. We know that this is sharp, due to an example isomorphic to $\mathcal J_2$ (it is also isomorphic to the corresponding example from Theorem~\ref{thmfull1}).

Since the right hand side of \eqref{eqfull1} strictly decreases as $l$ increases, we may conclude that for $k\ge 5$ any family $\ff$ with $\gamma(\ff)>\gamma_2 = 2$ will have strict inequality in \eqref{eqfull1} with $l=2$, and thus a strict inequality in \eqref{eqhk}. Moreover, for $k=4$ the lexicographic family with diversity $\gamma_2=n-4$ has the same cardinality as $\mathcal J_2$, displayed above (due to Theorem~\ref{thmfulleq}, or via direct calculation), and no family with $\gamma>\gamma_1$ can have larger cardinality.

Therefore, the bound \eqref{eqhk} is proven, and to complete the proof, we should only show that for $k\ge 5$ among the intersecting families $\ff\subset {[n]\choose k}$ with $\gamma(\ff)=2$ all families achieving equality in \eqref{eqhk} are isomorphic to $\mathcal J_2$. This could be done by a simple computation.

Any (maximal) intersecting family $\ff\subset {[n]\choose k}$ with $\gamma(\ff)=2$ is uniquely determined by the intersection of the two sets $A, B$, which are not containing the most popular element, and thus is isomorphic to one of the $\mathcal J_i$, $i\in[k]$.

 Using exclusion-inclusion formula for $\mathcal J_i$, we conclude that (the number of $k$-sets that contain 1 and do not intersect either $A$ or $B$) $=$ (the number of sets that contain 1 and miss $A$) $+$ (the number of sets that contain 1 and miss $B$) $-$ (the number of sets that contain 1 and miss both $A$ and $B$) $=2{n-k-1\choose k-1}-{n-2k+|A\cap B|-1\choose k-1}$, and this function clearly strictly increases as the intersection size of $A$ and $B$ decreases, and so we ``loose'' more and more sets containing $1$ as $i$ become smaller. Therefore, the unique (up to isomorphism) maximal intersecting family $\ff$ with $\gamma(\ff)=2$ is $\mathcal J_2$. Theorem is proven.
\end{proof}

Let us denote $\ff_{l}$ the maximal intersecting family with $\ff_l(\bar 1) = \mathcal L(l, k)$. Note that $\mathcal J_2$ is isomorphic to $\ff_2$. It is not difficult to see that, in terms of Theorem~\ref{thmfull2}, for $l = 0,\ldots, k-3$ we have $\ff_l(\bar 1) = \mathcal L(T_l, k)$ and, therefore, $\ff_l(1) = \mathcal L(S_l, k)$  (see the analysis after Theorem~\ref{thmfull2}). We also have $\ff_{n-k} = \mathcal L(T_{k-2},k)$. Moreover, it is not difficult to see that for $k-1<l<n-k$ we have $\ff_l\subset \ff_{n-k}$ (indeed, we have $\ff_l(1) = \ff_{n-k}(1)$ for this range). Also, using Theorems~\ref{thmfull2} and~\ref{thmfulleq} (or by direct calculation), we can conclude that $|\ff_{k-1}|<|\ff_{k-2}| = |\ff_{n-k}|$.

The next theorem is another application of our method.
\begin{thm}\label{thmclass} Assume that $k\ge 5$ and $n>2k$. Consider an intersecting family $\ff\subset{[n]\choose k}$, such that $\Delta(\ff) = \delta(1)$. Assume that  $|\cap_{F\in \ff(\bar 1)}F|\le k-2$. Then
\begin{equation} |\ff|\le |\mathcal J_3|,
\end{equation}
with the equality only possible if $\ff$ is isomorphic to $\mathcal J_3$.
\end{thm}
Note that $\mathcal J_3$ is in a sense the family with the smallest $\mathcal J_3(\bar 1)$ among the ones that satisfy the condition of the theorem.
Before we prove it, let us put it into context. The following theorem is one of the main results in \cite{KostM}:

\begin{thm}[\cite{KostM}]\label{thmko} Let $k\ge 5$ and $n=n(k)$ be sufficiently large. If $\ff\subset {[n]\choose k}$ is intersecting, and $|\ff|> |\mathcal J_3|$, then  $\ff\subset \ff_l$ for $l\in \{0,\ldots k-1, n-k\}$.
\end{thm}
Many results in extremal set theory are much easier to get once one assumes that $n$ is sufficiently large in comparison to $k$.
In this paper we prove the theorem above {\it without the restriction on $n$}.
It follows immediately from Theorem~\ref{thmclass}.
\begin{thm}\label{thmkonew} Let $k\ge 5$ and $n>2k$. If $\ff\subset {[n]\choose k}$ is intersecting, and $|\ff|\ge |\mathcal J_3|$, then either $\ff$ is isomorphic to $\mathcal J_3$, or to a subset of $\ff_l$ for $l\in \{0,\ldots k-1, n-k\}$.
\end{thm}
We note that $|\mathcal J_3|<|\ff_{n-k}|$ for $n = ck$: e.g., taking $n>4k$ is sufficient.\\

The following theorem is one of the main results of this paper. It generalizes Theorem~\ref{thmclass} and gives a reasonable classification of all large intersecting families (actually, all families with not too large diversity). We prove it using yet another variation of our methods. Let us call a family $\G\subset {[2,n]\choose k}$ \underline{typical minimal}, if, first, for any $G_l\in \G$ we have $|\cap_{G\in \G\setminus\{G_l\}}G|>|\cap_{G\in \G}G|$, and, second, either $|\G|=2$ or the number of elements contained in at least $2$ sets from $\G$ is strictly bigger than $|\G|$. (Note that the first condition implies that there are at least $|\G|$ elements contained in exactly $|\G|-1$ sets from $\G$. Therefore, the second condition is, in particular, satisfied when $|\cap_{G\in\G}G|>0$.)

\begin{thm}\label{thmclass2} Assume that $n>2k\ge 8$. Consider an intersecting family $\ff\subset{[n]\choose k}$, such that $\Delta(\ff) = \delta(1)$. Assume that $|\cap_{F\in \ff(\bar 1)}F|=r$, where $r\in [0,k-1]$. Choose any $k-1\ge t\ge r$ and any typical minimal subfamily $\mathcal G\subset \ff(\bar 1)$, such that $|\cap_{F\in \G}F|= t$. Take the (unique) maximal intersecting family $\ff'$, such that $\ff'(\bar 1) = \G$. Then, if $t\ge 4$ or $\gamma(\ff)\le {n-5\choose k-4}$, we have
\begin{equation}\label{eqclass1} |\ff|\le |\ff'|,
\end{equation}
and equality is possible if and only if $t=r$ and $\ff$ is isomorphic to $\ff'$.
In particular, in conditions above, if there are two sets in $A,B\in \ff(\bar 1)$, such that $|A\cap B|= k-s+1$ for some $s=2,\ldots, k$, then
\begin{equation} |\ff|\le |\mathcal J_s|
\end{equation}
with the equality only possible if $\ff$ is isomorphic to $\mathcal J_s$.
\end{thm}
We note that $r\ge 3$ implies both $t\ge 3$ and $\gamma(\ff)>{n-5\choose k-4}$, which, in turn, implies $|\ff|\le {n-1\choose k-1}-{n-5\choose k-1}+{n-5\choose k-4}$  due to Theorem~\ref{thm1}. Therefore, in particular, all families of size bigger than that are covered by the theorem.

\subsection{Proof of Theorem~\ref{thmclass}}
We are not going to use Theorem~\ref{thmfull2}, but rather give a self-contained proof in the spirit of the proofs in \cite{KZ}. The proof is rather simple and consists of two important steps. The first is, via shifting and rearranging the elements, to reduce the family in question to a family that has certain structure (that is, sets of particular form). The second step is to use the method in the spirit of Lemmas~\ref{lemdiv}, \ref{lemres}.

Let us first give the definitions related to {\it shifting}.\\

For a given pair of indices $1\le i<j\le n$ and a set $A \subset [n]$ define its \underline{$(i,j)$-shift} $S_{i,j}(A)$ as follows. If $i\in A$ or $j\notin A$, then $S_{i,j}(A) = A$. If $j\in A, i\notin A$, then $S_{i,j}(A) := (A-\{j\})\cup \{i\}$. That is, $S_{i,j}(A)$ is obtained from $A$  by replacing $j$ with $i$.

The  $(i,j)$-shift $S_{i,j}(\mathcal F)$ of a family $\mathcal F$ is as follows:

$$S_{i,j}(\mathcal F) := \{S_{i,j}(A): A\in \mathcal F\}\cup \{A: A,S_{i,j}(A)\in \mathcal F\}.$$

We call a family $\mathcal F$ \underline{shifted}, if $S_{i,j}(\mathcal F) = \mathcal F$ for all $1\le i<j\le n$.\\

We note that shifts do not destroy the cross-intersecting property, therefore, any pair of cross-intersecting families may be transformed into a pair of shifted cross-intersecting families.

Take the family $\ff$ satisfying the requirements of the statement. If $\gamma(\ff)=2$, then $\ff$ is isomorphic to one of the $\mathcal J_i$, $i=3,\ldots,k+1$, and we know that $\mathcal J_3$ is the largest out of them. Therefore, we may assume that $\gamma(\ff)\ge 3$.

Consider the families $\aaa:=\ff(1),\ \bb:=\ff(\bar 1)$. They are cross-intersecting. There are two cases to consider.\\

\textbf{Case 1. } Assume that we have two sets $A,B\in \ff(\bar 1)$, such that $|A\cap B|\le k-2$.

First of all, by doing shifts, we may assume that there are two sets $B_1,B_2\in \bb$, such that $|B_1\cap B_2|=k-2$. W.l.o.g., we may assume that $B_1 = [2,k-1]\cup \{k+1,k+2\}$, and that $B_2 = [2,k]\cup \{k+3\}$.

We have at least one more set $B_3$ in $\bb$. Doing $(i,j)$-shifts, where $i\in [2,k-1]$ and $j\in [k,n]$, we can assume that $B_3\supset [2,k-1]$. Apart from these elements, there are two more elements in $B_3$, say, $j_1$ and $j_2$, $j_1<j_2$. Either $j_1 = k$, or $j_2\ge k+3$, and we can do a $(k,j_2)$-shift (or a $(k,j_1)$-shift, if $j_1=k+3$), thus assuring that $[2,k]\in B_3$. Denote by $j$ the only remaining element of $B_3$. Clearly, $j\ne k+3$, since otherwise $B_3=B_2$, and therefore we can safely do a $(k+1,j)$-shift, thus transforming $B_3$ into $[2,k+1]$ and not changing $B_1,\ B_2$. Moreover, if $\gamma(\ff)\ge 4$, then there is yet another set $B_4$, which, in a similar way, we can transform into $[2,k]\cup \{k+2\}$. In what follows, we assume that $B_3 = [2,k+1]$ and $B_4 = [2,k]\cup \{k+2\}.$\\

Assume first that $\gamma(\ff)=|\bb|=3$ and $\bb = \{B_1,B_2,B_3\}$ as above. Then we can use the following simple argument. Remove the set $B_3$ from $\bb$, thus decreasing the sum of cardinalities of $\aaa$ and $\bb$ by $1$, and then add to $\aaa$ all the $(k-1)$-sets from the family
\begin{equation}\label{eqaaa}\aaa_1:= \big\{A\in {[2,n]\choose k-1}: A\cap [k+1]=\emptyset, \{k+2,k+3\}\in A\big\}.\end{equation}
Clearly, each set from $\aaa_1$ will intersect both $B_1$ and $B_2$, and thus the pair $\aaa, \bb$ will remain cross-intersecting. However, none of the sets from $\aaa_1$ was present in $\aaa$ before, since each set from $\aaa_1$ is disjoint with $B_3$. Therefore, we increase the sum of cardinalities of $\aaa$ and $\bb$ by $|\aaa_1| = {n-k-3\choose k-3}$, which is more than $1$ for $n> 2k$, $k\ge 4$. Therefore, the resulting cross-intersecting pair, which corresponds to the family $\mathcal J_3$, is bigger than the initial one.\\

If $\gamma(\ff)=|\bb|\ge 4$, then, due to the discussion two paragraphs above, we may assume that $B_i\in \bb$, $i\in[4]$, and that $\aaa, \bb$ are shifted.  Let $T$ be the lexicographically last element in $\bb$. since the family $\bb$ is itself intersecting, for each set $B\in \bb$ there must be an $i$ such that \begin{equation}\label{eq61}|[2,i]\cap B|>|[2,i]\setminus B|,\end{equation} because otherwise one of the sets in $[2,n]$ that may be obtained from $B$ by shifts will be disjoint from $B$, but at the same time it will lie in $\bb$.

Let us find the {\it biggest} such $i$ for $T$ and consider a bipartite graph $G$ with parts
\begin{align*}
\mathcal P_a:=&\big\{P\in {[2,n]\choose k-1}: P\cap [i] = [2,i]\setminus B\big\},\\
\mathcal P_b:=&\big\{P\in {[2,n]\choose k}: P\cap [i] = [2,i]\cap B\big\}
\end{align*}
and edges connecting disjoint sets. Then we see that, due to \eqref{eq61} and $n-1>2k-1$, $|\mathcal P_a|\ge |\mathcal P_b|$. At the same time, $(\aaa\cap \mathcal P_a)\cup(\bb\cap \mathcal P_b)$ is an independent set in $G$ and thus has size at most $|\mathcal P_a|$.
 Therefore, replacing $\aaa,\ \bb$ with $\aaa':=\aaa\cup \mathcal P_a$, $\bb':=\bb\setminus \mathcal P_b$, we do not decrease the sum of sizes. Moreover, it is easy to see that $\bb'$ is intersecting (since $\bb'\subset \bb$) and that $\aaa', \bb'$ are cross-intersecting. Indeed, any set in $\bb'$ must intersect $[2,i]$ in a set which lexicographically precedes $[2,i]\cap B$, and thus which intersects $[2,i]\setminus B$. Moreover, the family $\bb'$ is shifted.

 We repeat this with the lexicographically last element in $\bb'$, etc., until the lexicographically last set of the second family is $B_1$. We note that we will arrive at such a moment. Indeed, the only obstacle is that we somehow remove $B_1$ from the second family due to the fact that $B_1\in \mathcal P_b$, where $\mathcal P_b$ was defined by some other set $B'$ in the second family. Let us show that this is impossible. Recall that we chose $i$ as the largest index, for which the inequality \eqref{eq61} is satisfied. We must have $B'\cap [2,i] = B_1\cap [2,i]$. On the other hand, by the maximality of $i$, we have $1=|B'\cap [2,i]|-|B'\setminus [2,i]|=|B\cap [2,i]|-|B\setminus [2,i]|$. But $|B\cap [2,i]|-|B\setminus [2,i]|\ge 2$ for any $i\le k+2$.


Therefore, we may assume that $B_1$ is the lexicographically last element in $\bb$. Thus, each $B\in \bb,$ $B\ne B_1$, satisfies $[2,k]\subset B$. The number of such sets is $n-k$, and this, in particular, implies that $|\bb|\le n-k+1$. At the same time, $B_i\in \bb$, $i=[4]$, since they all precede $B_1$ lexicographically (and thus could not have been removed before $B_1$).

Removing all the sets from $\bb$ apart from $B_1,B_2$ will decrease $|\aaa|+|\bb|$ by at most $n-k-1$. At the same time, we may add to $\aaa$ the family $\aaa_1$ as well as the following family:

\begin{equation}\label{eqaaa2}\aaa_2:= \big\{A\in {[2,n]\choose k-1}: A\cap ([k]\cup\{k+2\})=\emptyset, \{k+1,k+3\}\in A\big\}.\end{equation}
We have $|\aaa_1|=|\aaa_2| = {n-k-3\choose k-3}\ge n-k-3$, and both $\aaa_1$ and $\aaa_2$ are disjoint with $\aaa$, since $B_3, B_4\in\bb$. Thus we are getting $2(n-k-3)$ new sets, and $2(n-k-3)-(n-k-1) = n-k-5> 0$ for $n>2k, k\ge 5$. Therefore, we conclude that in this case as well, the size of the family $\ff$ is smaller than the size of $\mathcal J_2$.\\

\textbf{Case 2. } Assume that any two sets in $\ff(\bar 1)$ intersect in at least $k-1$ elements. Then, knowing that $|\cap_{F\in \ff(\bar 1)}F|\le k-2$, we may assume that $\ff(\bar 1)\subset {[2,k+2]\choose k}$ (indeed, it is easy to check that all sets in $\ff(\bar 1)$ must be contained in a certain $k+1$-set). We may w.l.o.g. assume that the sets $C_1 = [2,k+1],\ C_2 = [2,k]\cup \{k+2\},\ C_3 = [2,k-1]\cup\{k+1,k+2\}$ are contained in $\ff(\bar 1)$. Let us look at the sets that $C_i$ forbid in $\ff(1)$, as compared to the sets that are forbidden by $B_1,B_2$. In both cases $\ff(1)$ contain all sets intersecting $[k-1]$, as well sets containing $\{k\}$ and one of $\{k+1,k+2\}$. Apart from the one listed above, the sets $C_i$ allow only for sets that intersect $[k+2]$ in $\{k+1,k+2\}$ (their number is ${n-k-2\choose k-3}$), while the sets $B_i$ allow for the sets that contain $\{k+3\}$, disjoint with $[k]$,  and intersect $[k+1,k+2]$ in at least one element (their number is ${n-k-2\choose k-3}+{n-k-3\choose k-3}$). Therefore, we can have ${n-k-3\choose k-3}$ sets more with $B_1,\ B_2$ than with $C_1,C_2,C_3$. On the other hand, in this case $|\ff(\bar 1)|-2\le k-1<{n-k-3\choose k-3}$ for any $n>2k\ge 10$. So we conclude that the families in this case are also smaller than $\mathcal J_2$.

\subsection{Proof of Theorem~\ref{thmclass2}}

Choose $\ff$, $t$, $\G$, and $\ff'$ satisfying the requirements of the theorem. We aim to prove that any family $\ff$ satisfying the requirements of the theorem has size strictly smaller than $\ff'$, unless it is isomorphic to $\ff'$. As a condition, in some of the cases we have $\gamma(\ff)\le {n-5\choose k-4}$. If it does not hold, but we have $t\ge 4$ instead, then we still may assume that $\gamma(\ff)\le {n-5\choose k-4}$. Indeed, the largest family with diversity $>{n-5\choose k-4}$ is at most as large as $\mathcal \mathcal H_5:=\{A\in {[n]\choose k}:[2,5]\subset A\}\cup\{A\in{[n]\choose k}: 1\in A, [2,5]\cap A\ne \emptyset\}$, and the latter family satisfies the requirements of the theorem (it contains a copy of any $\G$ as in the statement). Therefore, if we show that $\ff'$ is bigger than $\mathcal H_5$, it implies that $\ff'$ is bigger than any family with diversity $>{n-5\choose k-4}$.

Therefore, from now on we assume that $\gamma(\ff)\le {n-5\choose k-4}$. Assume that $\cap_{F\in \G}F = [2,t+1]$ (note that it may be empty).
For each $i\in [2,t+1]$ consider the following bipartite graph. The parts are
\begin{align*}
\mathcal P_a:=&\big\{P\setminus \{i\}: P\in \ff(1), i\in P\},\\
\mathcal P_b:=&\big\{P\in \ff(\bar 1): i\notin P\big\},
\end{align*}
and edges connect disjoint sets. Note that $\mathcal P_a\subset {X\choose k-2}$, $\mathcal P_b\subset {X\choose k}$, where $X = [2,n]\setminus \{i\}$, $|X|=n-2>k+k-2$. Due to the condition on $\gamma(\ff)$, we have $|\mathcal P_b\cap \ff(\bar 1)|\le {n-5\choose k-4}<{|X|-3\choose (k-2)-1}$. Thus, we can apply \eqref{eqcreasy} to
\begin{align*}
\aaa:=&\ff(1)\cap \mathcal P_a \ \ \ \ \text{and}\\
\bb:=&\ff(\bar 1)\cap \mathcal P_b,
\end{align*}
and conclude that $|\aaa|+|\bb|\le {|X|\choose k-2}$. Therefore, removing $\ff(\bar 1)\cap \mathcal P_b$ from $\ff(\bar 1)$ and adding sets from $\mathcal P_a$ to $\ff(1)$, we get a pair of families with larger sum of cardinalities. Moreover, the new pair is cross-intersecting: all sets in $\ff(\bar 1)\setminus \mathcal P_b$ contain $i$, as well as the sets from $\mathcal P_a$.
Therefore, we may assume that all sets in $\ff(\bar 1)$ contain $[2,t+1]$. \\

Put $\G = \{G_1,\ldots, G_z\}$. Due to minimality of $\G$, for each $G_l\in \G$, $l\in [z]$, there is $$i_l\in \Big(\bigcap_{G\in \G\setminus \{G_l\}}G\Big)\setminus \Big(\bigcap_{G\in \G}G\Big).$$
We assume that $i_l = t+l+1$, $l\in[z]$. In particular, $\{i_1,\ldots, i_z\} = [t+2,t+z+1]$.

Next, for each set $i\in [t+2, t+z+1]$ consider the same bipartite graph as before. We can apply \eqref{eqcreasy2} with $j:=k$. Indeed, we know that $|\bb|\ge {|X|\choose k-k}=1$ since $G_{i-t-1}\in \mathcal P_b$ (note that $G_l\notin \mathcal P_b$ for $l\ne i-t-1$, since all of them contain $i_l$ due to the definition of $i_l$). Therefore, $|\mathcal A|+|\mathcal B|\le {n-2\choose k-2}-{n-k-2\choose k-2}+1$, and we may replace $\aaa,\ \bb$ with $\bb':=\{B\}$ and $\aaa':=\{A\in {X\choose k-2}: A\cap B\ne \emptyset\}$. The resulting family is cross-intersecting. Repeating this procedure, we may conclude that any set in $\ff(\bar 1)$ not from $\G$ must contain the set $[2, t+z+1]$. If there are any other elements contained in all but 1 set from $\G$, we repeat the same procedure with them. Assume that they together with $[2,t+z+1]$ form a segment $[2,t']$. Note that for each $l \in [z]$ the set $G_l\cap [2,t']$ has the same size and must be non-empty. Otherwise, we have already proved the inequality \eqref{eqclass1}, since the current $\ff(\bar 1)$ is equal to $\G$.

Recall that $\G$ is typical minimal, that is, the number of elements contained in at least 2 sets from $\G$ is at least $z+1$. If $|[2,t']|\ge z+1$ (which is the case, e.g., when $t\ge 1$), then we proceed in the following way.\\

For each $S\subset[z]$ of size $z-1$ select one element $i_l\in G_l\cap [t'+1,n]$ for each $l\in S$. Note that $i_l$ may coincide. Put $I_S:=\{i_l:l\in S\}$. Consider a bipartite graph with parts
\begin{align*}
\mathcal P_a:=&\big\{P\setminus : P\in \ff(1), I_S\subset P, [2,t']\cap P=\emptyset\},\\
\mathcal P_b:=&\big\{P\setminus [2,t']: P\in \ff(\bar 1), I_S\cap P=\emptyset, [2,t']\subset P\big\},
\end{align*}
and edges connecting disjoint sets. Note that $\mathcal P_a\subset {Y\choose k-z'}$, $z'\le z$, and $\mathcal P_b\subset {Y\choose k-z''}$, $z''\ge z+1$, where $Y = [t'+1,n]\setminus I_S$, $|Y|=n-z''-z'$. In particular, $|Y|>k-z'+k-z''$. We have $|\G\cap \mathcal P_b| \in \{0,1\}$. Indeed, only the set $G_l,\ \{l\}=[z]\setminus S$ may be left. Denote
\begin{align*}
\aaa:=&\ff(1)\cap \mathcal P_a \ \ \ \ \text{and}\\
\bb:=&\ff(\bar 1)\cap \mathcal P_b.
\end{align*}
We have $k-z'>k-z''$, and, therefore, we may apply either \eqref{eqcreasy} or \eqref{eqcreasy2} with $a:=k-z',\ b:=k-z'',\ j:=k-z''$ (the upper bound on $|\bb|$ becomes trivial in that case) and conclude that either $|\aaa|+|\bb|\le {|Y|\choose k-z'}$ or $|\aaa|+|\bb|\le {|Y|\choose k-z'}-{|Y|-k+z''\choose k-z'}+1$. In both cases  we may replace $\bb$ with $\mathcal P_b\cap \G$ and $\aaa$ with all sets from $\mathcal P_a$ that intersect the set from
$\mathcal P_b\cap \G$ (if it is non-empty). As before, the size does not decrease and the cross-intersecting property is preserved.

Repeating this for all possible choices of $S$ and $I_S$, we arrive at a point when any set from $\ff(\bar 1)\setminus\G$ must intersect {\it any} set $I_S$. Tt is clear that it only holds for a set $F$ if $F\supset G_l\cap[t'+1,n]$. But then $F = G_l$, so $\ff(\bar 1) = \G$, and the proof of \eqref{eqclass1} is complete in this case. \\

Assume now that $|[2,t']|= z+1$. By the typical minimality of $\G$, we have an element $i\in$, which is contained in at least 2 sets, say, $G_1$ and $G_2$. We do something very similar to the previous case, but we have to be a bit more careful. We select a $z-1$-set $S\subset [z]$, which includes $1$ and $2$, a set $I_S:=\{i\}\cup\{i_l:l\in [3,z]\cap S\}$. It is clear that $|I_S|\le z-2$. Next, we consider the same bipartite graph as before. The sets in part $\mathcal P_a$ now have size at least $k-z+1$, while the sets in $\mathcal P_b$ have size at most $k-z$. Therefore, we can apply \eqref{eqcreasy}, \eqref{eqcreasy2} in the same way, and arrive at a family $\ff(\bar 1)$, such that any set in $\ff(\bar 1)\setminus \G$ intersects any set $I_S$ of the form described above. In practice, this means that any set in $\ff(\bar 1)\setminus \G$ contains $i$! Thus, we may now add $i$ to the set $[2,t']$ and proceed as in the first case: we now have at least $z+1$ common elements for all $F\in \ff(\bar 1)\setminus \G$. The inequality \eqref{eqclass1} is proven.

Finally, the uniqueness follows from the fact that the inequalities \eqref{eqcreasy}, \eqref{eqcreasy2} are strict unless the family $\bb$ has size $0$ and $1$, respectively. Therefore, if $\ff(\bar 1)\ne \G$, then at some point we would have had a strict inequality in the application of \eqref{eqcreasy}, \eqref{eqcreasy2}.

\section{Degree versions}\label{sec5}

Recently, Huang and Zhao \cite{HZ} gave an elegant proof of the following theorem using a linear-algebraic approach:

\begin{thm}[\cite{HZ}]\label{thmhz} Let $n>2k>0$. Then any intersecting family has
minimum degree at most ${n-2\choose k-2}$.
\end{thm}
The bound in the theorem is tight because of the trivial intersecting family, and the condition $n>2k$ is necessary: in \cite{HZ} the authors provide an example of a family for $n=2k$ which has larger minimal degree. In fact, for most values of $k$ there are {\it regular} intersecting families in ${[2k]\choose k}$ of maximal possible size: ${2k-1\choose k}$ (see \cite{IK}).  In the follow-up paper, Frankl, Han, Huang, and Zhao \cite{FHHZ} proved

\begin{thm}[\cite{FHHZ}]\label{thmfhhz} Let $k\ge 4$ and $n\ge ck^2$, where $c=30$ for $k=4,5$, and $c=4$ for $k\ge 6$. Then any non-trivial intersecting family has minimum degree at most ${n-2\choose k-2}-{n-k-2\choose k-2}$.
\end{thm}

 Several questions and problems arose in this context, that were asked in \cite{HZ}, \cite{FHHZ}, as well as in personal communication with Hao Huang and his presentation on the Recent Advances in Extremal Combinatorics Workshop at TSIMF, Sanya. Some of them are as follows:
\setlist{leftmargin=1cm}
\begin{enumerate}
\item Can one find a combinatorial proof of Theorem~\ref{thmhz}? This question was partially answered by Frankl and Tokushige \cite{FT6}, who proved it under the additional assumption $n\ge 3k$. Huang claims that their proof can be made to work for $n\ge 2k+3$, provided that one applies their approach more carefully. However, the cases $n=2k+2$ and $n=2k+1$ remained open.
\item Extend Theorem~\ref{thmfhhz} to the case $n\ge ck$ for large $k$. Ultimately, prove Theorem~\ref{thmfhhz} for all values $n\ge 2k+1$ for which it is valid.

\item Extend Theorems~\ref{thmhz} and \ref{thmfhhz} to degrees of $t$-tuples of vertices. The \underline{degree of a subset} $S\subset [n]$ is the number of sets from the family containing $S$. We denote by $\delta_t(\ff)$ the minimal degree of an $t$-element subset $S\subset [n]$.
\end{enumerate}

In this section we address these questions, partially answering all three of them. In the first theorem, we prove a $t$-degree version of Theorem~\ref{thmhz}. Its proof is combinatorial and works, in particular, for $s=1$ and $n\ge 2k+2$.

\begin{thm}\label{thm01} If $n\ge 2k+2>2$, then for any intersecting family $\ff$ of $k$-subsets of $[n]$ we have $\delta_1(\ff)\le {n-2\choose k-2}$. More generally, if $n\ge 2k+\frac{3t}{1-\frac tk}$ and $1\le t<k$, then $\delta_t(\ff)\le {n-t-1\choose k-t-1}$.\end{thm}

In the second theorem we give a $t$-degree version of Theorem~\ref{thmfhhz} with much weaker restrictions on $n$ for large $k$.
\begin{thm}\label{thm02} If $t=1$, $n\ge 2k+5$, and $k\ge 35$, or $1<t\le \frac k4-2$, $n\ge 2k+14t$, then for any non-trivial intersecting family $\ff$ of $k$-subsets of $[n]$ we have $\delta_t(\ff)\le {n-t-1\choose k-t-1}-{n-t-k-1\choose k-t-1}$. \end{thm}

After writing a preliminary version of the paper, we read the paper \cite{FT6}, where Theorem~\ref{thm01} is proven for $s=1$ and $n\ge 3k$. It turned out that the approach the authors took is very similar to the approach we use to prove Theorem~\ref{thm01}. However, it seems that their proof, unlike ours, does not work for $n=2k+2$, which is probably due to the fact that they use the original Frankl's degree theorem (see Section~\ref{sec2}).
\subsection{Calculations}\label{sec50}

In this section we do some of the calculations used in the proofs of Theorems~\ref{thm01} and~\ref{thm02}. Note that, substituting $u=3$ in \eqref{eq01}, we get that $$|\ff|\le {n-1\choose k-1}+{n-4\choose k-3}-{n-4\choose k-1} = {n-2\choose k-2}+{n-3\choose k-2}+{n-4\choose k-2}+{n-4\choose k-3}=$$ $${n-2\choose k-2}+2{n-3\choose k-2}=\Big[\frac {k-1}{n-1}+\frac{2(k-1)(n-k)}{(n-2)(n-1)}\Big]{n-1\choose k-1}=$$\begin{equation}\label{eq25}\frac{(k-1)(3n-2k-2)}{(n-1)(n-2)}{n-1\choose k-1}=\frac{k(k-1)(3n-2k-2)}{n(n-1)(n-2)}{n\choose k}.\end{equation}

We also have
$$\frac{{n-u-1\choose n-k-1}}{{n-u-1\choose k-1}}=\frac{\prod_{i=1}^{n-k-1}\frac{n-u-i}i} {\prod_{i=1}^{k-1}\frac{n-u-i}i} = \prod_{i=k}^{n-k-1}\frac{n-u-i}i= \prod_{i=k}^{n-k-1}\frac{n-u-i}{n-1-i}.$$
Clearly, for $3\le u\le k$ the last expression is maximized for $u\ge 3$, and we get the following bound, provided $n\ge 2k+2$:
\begin{equation}\label{eq04}\frac{{n-u-1\choose n-k-1}}{{n-u-1\choose k-1}}\ge \prod_{i=k}^{n-k-1}\frac{n-3-i}{n-1-i}= \frac{(k-1)(k-2)}{(n-k-1)(n-k-2)}\ \ \ \ \ \ \text{for }3\le u\le k.\end{equation}
If $n=2k+1$, then we get that the ratio is at least $\frac {k-3}{k-1}$.

We will also use the following formula:
\begin{equation}\label{eq07}
\frac{{n-t-k-1\choose k-t-1}}{{n-t-1\choose k-t-1}}=\prod_{i=1}^{k}\frac{n-k+1-i}{n-t-i}.
\end{equation}

\subsection{Proof of Theorem~\ref{thm01}}\label{sec51}

Take an intersecting family $\ff$ with maximum degree $\Delta$ and diversity $\gamma$. Then, by definition, $|\ff|= \Delta+\gamma$. W.l.o.g., we suppose that the element $1$ has the largest degree.
\begin{prop}\label{stat1} Fix some $n,t,k$. If for an intersecting family of $k$-sets $\ff\subset 2^{[n]}$ with maximum degree $\Delta$ and diversity $\gamma$  we have \begin{equation}\label{eq02}\Delta+\frac k{k-t}\gamma\le {n-1\choose k-1},\end{equation} then $\delta_t(\ff)\le {n-t-1\choose k-t-1}$.\end{prop}

\begin{proof}
The sum of $t$-degrees of all $t$-subsets of $[2,n]$ is $\gamma {k\choose t}+\Delta {k-1\choose t}.$ Therefore, there is a $t$-tuple $T$ of elements in $[2,n]$, such that
\begin{equation}\label{eq4} \delta_t(T)\le \frac{\gamma {k\choose t}+ \Delta {k-1\choose t}}{{n-1\choose t}}=\gamma\prod_{i=1}^{t}\frac {k-i+1}{n-i}+\Delta\prod_{i=1}^{t}\frac{k-i}{n-i}.\end{equation}
The ratio of two fractions is $$\frac{\prod_{i=1}^{t}\frac {k-i+1}{n-i}}{\prod_{i=1}^{t}\frac{k-i}{n-i}} =\prod_{i=1}^{t}\frac{k-i+1}{k-i}=\frac k{k-t}.$$
Therefore, if \eqref{eq02} holds, then
\begin{equation}\label{eq5}\delta_t(\ff)\le \prod_{i=1}^{t}\frac{k-i}{n-i}(\Delta+\frac{k}{k-t}\gamma)\le\prod_{i=1}^{t}\frac{k-i}{n-i}{n-1\choose k-1}={n-t-1\choose k-t-1}.\end{equation}

\end{proof}

To prove Theorem~\ref{thm01}, we are only left to verify \eqref{eq02} for all intersecting families. It is vacuously true for trivial intersecting families, so we may assume that $\gamma\ge 1$. Fix $\ff$ as in the theorem. We have two cases to distinguish.\\

\textbf{Case 1. $\gamma\le {n-4\choose k-3}$. } We only need to show that \eqref{eq01} holds. We may apply Corollary~\ref{corweight} (otherwise, it is not difficult to obtain via direct calculations). We only have to check that \begin{equation}\label{eq151}\frac{k}{k-t}\le \frac{n-k-3}{k-3}.\end{equation} Putting $n = 2k+s$, where $s\ge 1$, we see that, if $t=1$, then \eqref{eq151} holds for any $s\ge 1$. If $t>1$, then we must have
$$k^2-3k-(k-3)t+s(k-t)\ge k^2-3k,$$
which is satisfied when $$s\ge \frac{k-3}{k-t}t \ \ \ \Leftarrow \ \ \  s\ge \frac{t}{1-\frac tk}.$$



\textbf{Case 2. $\gamma\ge {n-4\choose k-3}$. } By the calculations in Section~\ref{sec50}, We know that $|\ff|\le {n-2\choose k-2}+2{n-3\choose k-2}$, and we use the following easy bound on $\delta_t(\ff):$ \begin{equation}\label{eq06}\delta_t(\ff)\le \frac{{k\choose t}}{{n\choose t}} |\ff|.\end{equation}
Thus, it is sufficient for us to check that the following inequality holds:

\begin{equation}\label{eq05}{n-t-1\choose k-t-1}\ge \frac{{k\choose t}}{{n\choose t}}\Big[{n-2\choose k-2}+2{n-3\choose k-2}\Big].\end{equation}
We have
\begin{equation}\label{eq26} \frac{{n-t-1\choose k-t-1}{n\choose t}}{{k\choose t}} = \frac{(n-t-1)!n!(k-t)!}{(k-t-1)!(n-k)!(n-t)!k!}=\frac {k-t}{n-t}{n\choose k}\end{equation}
and, by the calculation in Section~\ref{sec50}, $${n-2\choose k-2}+2{n-3\choose k-2}=\frac{k(k-1)(3n-2k-2)}{n(n-1)(n-2)}{n\choose k}.$$
Therefore, \eqref{eq05} is equivalent to
$$\frac{k-t}{n-t}\ge \frac{k(k-1)(3n-2k-2)}{n(n-1)(n-2)}.$$

If $t=1$, then it simplifies to $\frac{k(3n-2k-2)}{n(n-2)}\le 1$, which holds for any $n\ge 2k+2$.


If $t>1$, then it simplifies to a quadratic inequality in $n$, which holds for $$n\ge\frac{2k^2+(k-3)t+\sqrt{(2k^2+(k-3)t)^2 -8(k-t)t(k^2-1)}}{2(k-t)}.$$
The right hand side is at most
$$\frac{2k^2+(k-3)t+\sqrt{(2k^2+(k-3)t)^2}}{2(k-t)}= \frac{2k^2+(k-3)t}{(k-t)} =2k+\frac{3(k-1)t}{k-t}<2k+\frac {3t}{1-\frac tk}.$$

\subsection{Proof of Theorem~\ref{thm02}}\label{sec52}
The strategy of the proof is very similar to that of Theorem~\ref{thm01}. Fix a non-trivial intersecting family $\ff$ with $\gamma\ge 2$ (otherwise, it is a subfamily of some Hilton-Milner family). W.l.o.g., suppose that the element of the family $\ff$ with maximal degree is 1, and that $\ff$ contains the set $[2,k+1]$. Then any other set containing $1$ must intersect $U$. We compare $\ff$ with the Hilton-Milner family $\mathcal {HM}:=\{F:1\in F, F\cap [2,k+1]\ne \emptyset\} \cup [2,k+1].$ We consider cases depending on $\gamma$.
The case analysis, however, will be more complicated, as compared to the previous case. Notably, we get a new non-trivial Case 1.\\

\textbf{Case 1. $1<\gamma<{n-k+t+1\choose t+2}$. }
 We have $\gamma\ge 2$, and we may apply Theorem~\ref{thmHil} to the families $\ff(1):=\{F\setminus\{1\}:1\in F\in\ff\}$ and $\ff(\bar 1):=\{F:1\notin F\in\ff\}$. We get that the cardinality of $\ff(1)$ is at most the cardinality of all $(k-1)$-subsets of $[2,n]$ that intersect the sets $[2,k+1]$ and $[2,k]\cup \{k+2\}$. In other words, the degree of $1$ is at most ${n-1\choose k-1}-{n-k-1\choose k-1}-{n-k-2\choose k-2}$. That is, some ${n-k-2\choose k-2}$ sets from $\mathcal {HM}$ containing $1$ are missing from $\ff$.

Denote $\mathcal G:=\{G\in{[n]\choose k}\setminus\ff:1\in G, G\cap [2,k+1]\ne \emptyset\}$. By the paragraph above, we have $|\mathcal G| \ge {n-k-2\choose k-2}$. Consider a subfamily $\mathcal G'\subset \mathcal G$, such that each $G'\in\mathcal G'$ intersects $[k+2,n]$ in at least $t$ elements.

Then, denoting by $\mathcal H$ the Hilton-Milner family, it is easy to see that the following rough estimate holds. \begin{equation}\label{eq18}\delta_t(\mathcal H)-\delta_t(\ff)\ge \frac{|\mathcal G'|}{{n-k-1\choose t}}-\gamma.
\end{equation}
Indeed, the first summand is the average loss in the $t$-degree of $t$-sets in $[k+2,n]$, and each set contributing to $\gamma$ can contribute at most 1 to minimum $t$-degree. In the rest of this case our goal is to show that the RHS in \eqref{eq18} is always positive.

Let us first estimate the size of $\mathcal G'$. To do so, we have to exclude all the sets that intersect $[1,k+1]$ in more than $k-t$ elements. Since $1$ is always in the set, the number of sets we have to exclude is $$\sum_{i=k-t}^{k-1}{k\choose i}{n-k-1\choose k-i-1}.$$
We have ${n-k-1\choose t-j}>2{n-k-1\choose t-j-1}$ for any $j\ge0$, since $n\ge 2k\ge 8t$.  Therefore, the sum above is at most $${k\choose k-t}{n-k-1\choose t}={k\choose t}{n-k-1\choose t},$$
and we have \begin{equation}\label{eq184}|\mathcal G'|\ge {n-k-2\choose k-2}-{k\choose t}{n-k-1\choose t}.\end{equation}
To show that the RHS in \eqref{eq18} is always positive and thus to conclude the proof in Case 1, it is sufficient to show the first in the following chain of inequalities:
\begin{small}\begin{multline}\label{eq19} {n-k-2\choose k-2}\ge {n-k-1\choose t}{n-k+t+2\choose t+2}>\\ >{n-k-1\choose t}{n-k+t+1\choose t+2}+{k\choose t}{n-k-1\choose t}.\end{multline}\end{small}
Note that the second inequality is just a corollary of Newton's binom and the fact that ${k\choose t}<{n-k+t+1\choose t+1}$. We also use the fact that $\gamma\le {n-k+t+1\choose t+2}$. \\

If $t=1$, $n\ge 2k+5$, then in the worst case for \eqref{eq19} is $n=2k+5$, and the first inequality in \eqref{eq19} transforms into $${k+3\choose 5}\ge (k+4){k+8\choose 3},$$
which holds for $k\ge 35$.\\

If $1<t\le \frac k{4}-2$, $n\ge 2k+14t$, then we have $${n-k-1\choose t}{n-k+t+2\choose t+2}\le \frac{(n-k+t+2)!}{t!(t+2)!(n-k-t)!}={n-k+t+2\choose 2t+2}{2t+2\choose t}\le$$$$ {n-k+t+2\choose 2t+2}2^{2t},$$
 where the last inequality holds for any $t\ge 2$.
 On the other hand, using the above conditions on $n,k,t$, we have $${n-k+t+2\choose 2t+2}< \Big(\frac{n-k-2}{n-k-2t-4}\Big)^{t+4}{n-k-2\choose 2t+2}\le $$
 $$\Big(1+\frac {2t+2}{k+10t-4}\Big)^{t+4}{n-k-2\choose 2t+2}\le \Big(\frac{19} {16}\Big)^{t+4}{n-k-2\choose 2t+2}< $$
 $$ \Big(\frac 54\Big)^{t+4} \Big(\frac{4t+4}{n-k-4t-6}\Big)^{2t+2}{n-k-2\choose 4t+4}\le \Big(\frac 54\Big)^{t+4} \Big(\frac{4t+4}{14t+2}\Big)^{2t+2}{n-k-2\choose 4t+4}\le$$ $$\Big(\frac 54\Big)^{t+4} \Big(\frac{2}{5}\Big)^{2t+2}{n-k-2\choose 4t+4}\le \Big(\frac{1}{2}\Big)^{2t+2}{n-k-2\choose k-2}.$$
 Comparing this chain of inequalities with the one above, we conclude that for our choice of parameters

$$\frac{{n-k-2\choose k-2}}{{n-k-1\choose t}{n-k+t+2\choose t+2}}\ge \frac {2^{2t+2}}{2^{2t}}>1.$$






\vskip+0.2cm


\textbf{Case 2. ${n-k+t+1\choose t+2}\le \gamma\le {n-4\choose k-3}$. }
Using the first inequality in \eqref{eq5}, we get the  following analogue of Statement~\ref{stat1}.
\begin{stat}Fix some $n,t,k$. If for an intersecting family of $k$-sets $\ff\subset 2^{[n]}$ with maximum degree $\Delta$ and diversity $\gamma$  we have \begin{equation}\label{eq12}\Delta+c_t\gamma\le \Big(1-\prod_{i=1}^{k}\frac{n-k+1-i}{n-t-i}\Big){n-1\choose k-1},\end{equation} then $\delta_t(\ff)\le {n-t-1\choose k-t-1}-{n-t-k-1\choose k-t-1}$.\end{stat}
\begin{proof}
Indeed, if \eqref{eq12} holds, then, using \eqref{eq5}, we get
$$\delta_t(\ff)\le \prod_{i=1}^{t}\frac{k-i}{n-i} \Big(1-\prod_{i=1}^{k}\frac{n-k+1-i}{n-t-i}\Big){n-1\choose k-1} =$$$$ {n-t-1\choose k-t-1}\Big(1-\prod_{i=1}^{k}\frac{n-k+1-i}{n-t-i}\Big) \overset{\eqref{eq07}}{=} {n-t-1\choose k-t-1}-{n-t-k-1\choose k-t-1}.$$
\end{proof}

We apply \eqref{eq01}. Our situation corresponds to the case $u\le k-t-2$. To verify \eqref{eq12}, one has to check that \begin{equation}\label{eq13}{n-u-1\choose k-1}-c_t{n-u-1\choose n-k-1}-\prod_{i=1}^{k}\frac{n-k+1-i}{n-t-i}{n-1\choose k-1}\ge 0.\end{equation} We have
$${n-u-1\choose k-1}-c_t{n-u-1\choose n-k-1} = \Big(1-\frac{k}{k-t}\prod_{i=k}^{n-k-1}\frac{n-u-i}{n-1-i}\Big) {n-u-1\choose k-1}\overset{\eqref{eq04}}{\ge}$$
$$\Big(1- \frac{k(k-1)(k-2)}{(k-t)(n-k-1)(n-k-2)}\Big) \prod_{i=1}^{k-1}\frac{n-u-i}{n-i}{n-1\choose k-1}.$$
The last expression is minimized when $u=k-t-2$. Comparing the product above with the product in \eqref{eq13}, we get $$\frac{\prod_{i=1}^{k}\frac{n-k+1-i}{n-t-i}} {\prod_{i=1}^{k-1}\frac{n-u-i}{n-i}}\le \frac{\prod_{i=1}^{k}\frac{n-k+1-i}{n-t-i}} {\prod_{i=1}^{k-1}\frac{n-k+t+2-i}{n-i}}= \frac{\prod_{i=1}^{k-t-1}\frac{n-k-t-i}{n-t-i}} {\prod_{i=1}^{k-t-2}\frac{n-k-i}{n-i}}\le $$$$ \frac{\prod_{i=1}^{k-t-1}\frac{n-k-t-i}{n-t-i}} {\prod_{i=1}^{k-t-2}\frac{n-k-t-i}{n-t-i}}=1-\frac{k}{n-k+1}.$$
Therefore, to prove \eqref{eq13}, it is sufficient for us to show that \begin{equation}\label{eq16} 1- \frac{k(k-1)(k-2)}{(k-t)(n-k-1)(n-k-2)}\ge 1-\frac{k}{n-k+1}.\end{equation}
For the fraction in the left hand side, we use the following property: if we add 1 to one of the multiples in the numerator and 1 to one of the multiples in the denominator, then the fraction will only increase, and the expression in the left hand side will decrease. If $t=1$, then the LHS of \eqref{eq16} is $$1- \frac{k(k-2)}{(n-k-1)(n-k-2)}\ge 1-\frac{k^2}{(n-k+1)(n-k-2)}>1-\frac{k}{n-k+1},$$
and therefore \eqref{eq13} is satisfied for any $k$ and
$n\ge 2k+2$.

If $1<t\le \frac k4$ and $n\ge 2k+4t$, then $(k-t)(n-k-1)\ge (k-t)(k+4t-1)= k^2+3kt-4t^2-k+t>k^2$, and, therefore,
the LHS of \eqref{eq16} is at least $$1- \frac{k^3}{(k-t)(n-k-1)(n-k+1)}> 1-\frac{k}{n-k+1},$$
and \eqref{eq13} is satisfied again.\\

\textbf{Case 3. $\gamma\ge {n-4\choose k-3}$. } As in Case 2 of the proof of Theorem~\ref{thm01}, we know that \eqref{eq06} holds. We have to verify that

\begin{equation}\label{eq08}{n-t-1\choose k-t-1}-{n-t-k-1\choose k-t-1}\ge \frac{{k\choose t}}{{n\choose t}}\Big[{n-2\choose k-2}+2{n-3\choose k-2}\Big]\end{equation}
holds.

Using \eqref{eq25}, \eqref{eq07} and \eqref{eq26}, the inequality \eqref{eq08} is equivalent to
\begin{equation}\label{eq11}\frac{k-t}{n-t} \Big(1-\prod_{i=1}^{k}\frac{n-k+1-i}{n-t-i}\Big)\ge \frac{k(k-1)(3n-2k-2)}{n(n-1)(n-2)}.\end{equation}

If $t=1$, then it simplifies to $$\Big(1-\prod_{i=1}^{k}\frac{n-k+1-i}{n-1-i}\Big)\ge \frac{k(3n-2k-2)}{n(n-2)} \ \ \Leftrightarrow \ \ \frac{(n-k)(n-2k-2)}{n(n-2)}\ge \prod_{i=1}^{k}\frac{n-k+1-i}{n-1-i}.$$
This is equivalent to \begin{equation}\label{eq09}\frac{(n-2k-2)}{n}\ge \prod_{i=2}^{k}\frac{n-k+1-i}{n-1-i}.\end{equation}
The right hand side is at most $$\frac{(n-2k+1)(n-2k+2)}{(n-3)(n-4)}\le \frac{(n-2k+4)(n-2k+2)}{n(n-4)}.$$ Therefore \eqref{eq09} follows from
$$\frac{n-2k-2}{n}\ge \frac{(n-2k+4)(n-2k+2)}{n(n-4)}\ \ \Leftrightarrow (n-2k-2)(n-4)\ge (n-2k+4)(n-2k+2),$$ which holds for any $n\ge 2k+4$ and $k\ge 12.$

If $1<t\le \frac k4-2$, then

$$1-\frac{n-t}{k-t}\cdot\frac{k(k-1)(3n-2k-2)}{n(n-1)(n-2)}\ge 1-\frac{k(k-1)(3n-2k-2)}{(k-t)n(n-1)}=$$
$$\frac{(k-t)n^2-(3k^2-3k-k+t)n+2(k^3-k)}{(k-t)n(n-1)}\ge \frac{(k-t)n^2-3k^2n+2k^3}{(k-t)n(n-1)}\overset{(*)}{\ge}$$$$ \frac{(n-k)(n-2\frac {k^2+kt}{k-t})}{n(n-1)}\ge \frac{(n-k)(n-2k-\frac {2kt}{k-t})}{n^2}\ge \frac{(n-k)(n-2k-3t)}{n^2},$$
where in the inequality (*) we used the fact that the difference between the first numerator and the second numerator multiplied by by $(k-t)$ is $ktn-2k^2t$, which is positive for $n>2k$; in the last inequality we used that $t\le k/3$. On the other hand,
$$\prod_{i=1}^{k}\frac{n-k+1-i}{n-t-i}=\prod_{i=1}^{k-t-1} \frac{n-t-k-i}{n-t-i}\le \Big(\frac{n-k}{n}\Big)^{k-t-1}\le \Big(\frac {n-k}n\Big)^{\frac {3k}{4}+1}.$$
Therefore, combining these calculations, the inequality \eqref{eq11} would follow from the inequality
\begin{equation}\label{eq10}1-\frac{2k+3t}n\ge \Big(1-\frac k{n}\Big)^{3k/4}.\end{equation}
 If $2k+14t\le n\le7k$, then the right hand side of the inequality above is at most $ e^{-\frac{3k^2}{4n}}<e^{-\frac {k}{10}},$ while the left hand side is at least $\frac {11t}{2k+14t}>\frac {22}{2k+28}$. It is easy to see that, say, for $k\ge 10$, $\frac {22}{2k+28}>e^{-\frac {k}{10}}$.

If $n>7k$ and $k\ge10$, then $$\Big(1-\frac k{n}\Big)^{3k/4}<\Big(1-\frac k{n}\Big)^{7}<1-\frac {7k}n+\frac{21k^2}{n^2}\le 1-\frac {4k}n.$$
On the other hand, $1-\frac {2k+3t}n\ge 1-\frac {3k}n$, therefore, the inequality \eqref{eq10} is verified in this case.\\

To conclude, we remark that the only conditions on $k$ that we used for $t\ge 2$ were $k\ge 4t+8$ and $k\ge 10$. The later one is implied by the former one.

\section{Degree version of the Erd\H os Matching Conjecture}\label{sec6}
The \underline{matching number} $\nu(\ff)$ of a family $\ff$ is the maximum number of pairwise disjoint sets from $\ff$. That is, intersecting families are exactly the families with matching number one. It is a natural question to ask, what is the largest family with matching number (at most) $s$. Speaking of uniform families, let us denote $e_k(n,s)$ the size of the largest family $\ff\subset{[n]\choose k}$ with $\nu(\ff)\le s$. Note that this question is only interesting when $n\ge k(s+1)$.  The following two families are the natural candidates: $$\mathcal A_0(n,k,s):=\{A\subset[n]: A\cap [s]\ne \emptyset\},$$
$$A_k(k,s):={[k(s+1)-1]\choose k}.$$
Erd\H os conjectured \cite{E} that $e_k(n,s)=\max\big\{|\aaa_0(n,k,s)|,|\aaa_k(k,s)|\big\}$. This conjecture is known as the Er\H os Matching conjecture. It was studied quite extensively over the last 50 years, but it remains unsolved in general.  It is known to be true for $k\le 3$ \cite{F5} and for $n\ge (2s+1)(k-1)$ \cite{F4}. We note that $\aaa_0(n,k,s)$ is bigger than $\aaa_k(k,s)$ already for relatively small $n$: the condition $n>(k+1)(s+1)$ should suffice.

A degree version of Erd\H os Matching Conjecture and related problems attracted a lot of attention recently (see, e.g., \cite{HPS}, \cite{KO}).
The following theorem was proved in \cite{HZ}.
\begin{thm}[\cite{HZ}]\label{thmhzm} Given $n,k,s$ with $n\ge 3k^2(s+1)$, if for a family $\ff\subset{[n]\choose k}$ with $\nu(\ff)\le s$ we have $\delta_1(\ff)\le {n-1\choose k-1}-{n-s-1\choose k-1}=\delta_1(\aaa_0(n,k,s))$.
\end{thm}
This improved the result of Bollob\'as, Daykin, and Erd\H os \cite{BDE}, who arrived at the same conclusion for $n\ge 2k^3s$. The authors conjectured that the same should hold for any $n>k(s+1)$.
Note that in the degree version we do not include the family $\aaa_k(k,s)$, since its minimum $t$-degree is 0 for $n\ge k(s+1)$ and $t\ge 1$.  Note that, for general $t$ we have $\delta_t(\aaa_0(n,k,s))= {n-t\choose k-t}-{n-s-t\choose k-t}$.

In this paper we improve and generalize Theorem~\ref{thmhzm} above result for $k$ large in comparison to $s$.

\begin{thm}\label{degEKR} Fix $n,s,k$ and $t\ge 1$, such that $n\ge 2k^2$,  and $k\ge 5st$ ($k\ge 3s$ for $t=1$). For any family $\ff\subset {[n]\choose k}$ with $\nu(\ff)\le s$ we have $\delta_t(\ff) \le \delta(\aaa_0(n,k,s))$, with equality only in the case $\ff=\aaa_0(n,k,s)$.
\end{thm}
We note that the constants in the proof are not optimal, and chosen in this way to simplify the calculations.
In the proof we make use of the stability theorem, proved by the author together with P. Frankl \cite{FK7, FK8}. Recall that $\tau(\ff)$ is the minimal size of a set $S\subset [n]$, such that $S\cap F\ne \emptyset$ for any $F\in \ff$. For fixed $n,s,k$, saying that $\ff\subset {[n]\choose k}$ satisfies $\nu(\ff)\le s$ and $\tau(\ff)>s$ is equivalent to saying that $\ff$ is not isomorphic to a subfamily of $\mathcal A_0(n,k,s)$.

\begin{thm}[\cite{FK7}]\label{thmhil2} Let $n = (u+s-1)(k-1)+s+k,$ $u\ge s+1$. Then for any family $\mathcal G\subset {[n]\choose k}$ with $\nu(\mathcal G)= s$ and $\tau(\mathcal G)\ge s+1$ we have
\begin{equation}\label{eqhil} |\mathcal G|\le {n\choose k}-{n-s\choose k} - \frac{u-s-1}u{n-s-k\choose k-1}.
\end{equation} \end{thm}\vskip+0.4cm

\begin{proof}[Proof of Theorem \ref{degEKR}]
We prove the theorem by contradiction. Fix $t\ge 1$ and take a family $\ff$ with $\delta_t(\ff)>{n-t\choose s-t}-{n-s-t\choose s-t}$. It cannot be a subfamily of $\aaa_0(n,k,s)$ since $\delta_t(\ff)>\delta_t(\aaa_0(n,k,s))$. Therefore, $\tau(\ff)\ge s+1$ and, assuming that $\nu(\ff)\le s$, we conclude that $|\ff|$ satisfies \eqref{eqhil}. By simple double counting, we have
$$\delta_t(\ff)\le \frac{{k\choose t}}{{n\choose t}}\Big[{n\choose k}-{n-s\choose k} - \frac{u-s-1}u{n-s-k\choose k-1}\Big].$$
Note that ${n-s-t\choose k-t}=\frac{{k\choose t}}{{n-s\choose t}}{n-s\choose k}$ and ${n-t\choose k-t}=\frac{{k\choose t}}{{n\choose t}}{n\choose k}$. We have
$$\delta_t(\aaa_0(n,k,s))-\delta_t(\ff)\ge \frac{{k\choose t}(u-s-1)}{{n\choose t}u}{n-s-k\choose k-1}-\Big[\frac{{k\choose t}}{{n-s\choose t}}-\frac{{k\choose t}}{{n\choose t}}\Big]{n-s\choose k}=$$
$$\frac{{k\choose t}}{{n\choose t}}\Bigg[\frac{u-s-1}{u}{n-s-k\choose k-1}-\Big[\prod_{i=0}^{t-1} \frac{n-i}{n-s-i}-1\Big]{n-s\choose k}\Bigg]=(*).$$
We have $\prod_{i=0}^{t-1} \frac{n-i}{n-s-i}-1\le (1+\frac s{n-s-t})^t-1.$ It is not difficult to verify that for $\theta<\frac 1{2m}$ one has $(1+\theta)^m\le 1+2m\theta$. Therefore, assuming that \begin{equation}\label{eq33}n\ge s+t+2st,\end{equation} we have \begin{equation}\label{eq35}\prod_{i=0}^{t-1} \frac{n-i}{n-s-i}-1\le \frac {2ts}{n-s-t}.\end{equation}
On the other hand, we have $(1-\theta)^m\ge 1-m\theta$ and $t\le k-1$, and $$\frac{{n-s-k\choose k-1}}{{n-s\choose k}}=\frac {k}{n-s-k+1}\prod_{i=0}^{k-1}\frac{n-s-k-i+1}{n-s-i}>\frac{k}{n-s-t} \Big(1-\frac {k-1}{n-s-k}\Big)^k>$$
$$\frac{k}{n-s-t}\Big(1-\frac{(k-1)k}{n-s-k}\Big)>\frac k{2(n-s-t)},$$
provided \begin{equation}\label{eq34}n\ge s+k+2k(k-1).\end{equation} We conclude that, provided that \eqref{eq33} and \eqref{eq34} hold, we get
$$(*)> \frac{{k\choose t}}{{n\choose t}}\Bigg[\frac{u-s-1}{u}\frac k{2(n-s-t)}-\frac {2ts}{n-s-t}\Bigg]{n-s\choose k},$$
which is nonnegative provided $k\ge 4ts\frac u{u-s-1}$. This inequality holds for $k\ge 5ts$ and $u\ge 9s$. The last inequality is satisfied for $n\ge 2k^2$, since then $n\ge 2k^2\ge (9s+s)k\ge (9s+s-1)(k-1)+s+k$. We note that with this choice of $n$ and $k$ both \eqref{eq33} and \eqref{eq34} hold.

For $t=1$ one may improve \eqref{eq35} to $\frac n{n-s}-1\le \frac s{n-s}$ and the condition on $k$ may be relaxed to $k\ge 3s$. The equality part of the statement follows easily from the fact that strict inequality is obtained in the case when $\tau(\ff)\ge s+1$.
\end{proof}

\section{Conclusion}\label{sec7}
In this paper we explored several question concerning intersecting families. Some of these questions remain only partially resolved, and it would be highly desirable to settle them. 

First of all, it would be desirable to understand the structure of families with diversity larger than ${n-3\choose k-2}$. As shown in \cite{Kup21}, no such family exist for $n>Ck$ with some absolute constant $C$. It is believed to be true for $n>3k-2$. For $2k<n<3k$, however, we do have such families, and both the proof of the conjecture above and the understanding of their structure is interesting on its own and would be helpful in different questions (for more information, see \cite{Kup21}).

In this paper we have applied the machinery of Section~\ref{sec3}, based on papers \cite{FK1}, \cite{KZ}, to the setting of general families (see Theorem~\ref{thmclass2}). This gave a reasonable classification of {\it all} large intersecting families. However, it is by no means complete.

\begin{pro} To what extent one could relax the condition on diversity and/or on $t$ in Theorem~\ref{thmclass2}?
\end{pro}
\begin{pro} In terms of Theorem~\ref{thmclass2}, is there a reasonable way to compare the sizes of intersecting families generated by typical minimal families? In particular, we believe that, if $\ff(\bar 1)$ contains a typical minimal subfamily $\mathcal G$, such that $|\cap_{F\in \G}F|= t\ge 5$, then
\begin{equation} |\ff|\le |\mathcal J_{k-t+1}|
\end{equation}
with equality only possible if $\ff$ is isomorphic to $\mathcal J_s$.
\end{pro}
We believe that Theorem~\ref{thmclass2} is an important step towards classification of families with large covering numbers $\tau(\ff)$ (the size of the smallest subset that intersects any set from $\ff$). An important result in this direction we obtained by Frankl \cite{F8}, who resolved this problem for $\tau(\ff)=3$ and, again, for large enough $n = n(k)$.
\begin{pro} Extend Theorem~\ref{thmclass2} to families with $\tau(\ff)\ge k$.
\end{pro}

The next question concerns the degree version of the Hilton-Minler theorem.

\begin{pro} Is there an example for $n=2k+1$ ($n=2k+ct$, $c$ is a small constant), such that there exists a non-trivial intersecting family $\ff$ with minimal $1$-degree ($t$-degree) higher than that of the Hilton-Milner family?
\end{pro}
One reason to believe that the answer to this question is positive is that the degrees of elements in the Hilton-Milner family are irregular, even if we exclude the element of the highest degree out of consideration.

Finally, the following question concerning cross-intersecting families seem interesting for us.

\begin{pro} Consider two cross-intersecting families $\aaa,\ \bb\subset{[n]\choose k}$ that are disjoint. Is it true that $$\min\{|\aaa|, |\bb|\}\le \frac 12{n-1\choose k-1}?$$
\end{pro}


\begin{thebibliography}{111}
\bibitem{BDE} B. Bollob\'as, D.E. Daykin, P. Erd\H os, \textit{Sets of independent edges of a hypergraph}, Quart. J. Math. Oxford Ser. 27 (1976), N2, 25--32.

\bibitem{E} P. Erd\H os, \textit{A problem on independent r-tuples}, Ann. Univ. Sci. Budapest. 8 (1965) 93--95.

\bibitem{EKR} P. Erd\H os, C. Ko, R. Rado, \textit{Intersection theorems for systems of finite sets}, The Quarterly Journal of Mathematics, 12 (1961) N1, 313--320.

\bibitem{F8} P. Frankl, {\it On intersecting families of finite sets}, Bull. Austral. Math. Soc. 21 (1980), 363–
372.
\bibitem{Fra1} P. Frankl,  \textit{Erdos-Ko-Rado theorem with conditions on the maximal degree}, Journal of Combinatorial Theory, Series A 46 (1987), N2, 252--263.

\bibitem{F4} P. Frankl, \textit{Improved bounds for Erd\H os' Matching Conjecture}, Journ. of Comb. Theory Ser. A 120 (2013), 1068--1072.

\bibitem{F5} P. Frankl, {\it On the maximum number of edges in a hypergraph with given matching number}, Discrete
Applied Mathematics 216 (2017), 562--581.
\bibitem{FHHZ} P. Frankl, J. Han, H. Huang, Y. Zhao {\it A degree version of Hilton-Milner theorem}, arXiv:1703.03896v2

\bibitem{FK1} P. Frankl, A. Kupavskii, {\it Erd\H os-Ko-Rado theorem for $\{0, \pm 1\}$-vectors}, arXiv:1510.03912

\bibitem{FK7} P. Frankl, A. Kupavskii, \textit{Families with no $s$ pairwise disjoint sets}, Journal of London Mathematical Society (2017),   arXiv:1607.06122

\bibitem{FK8} P. Frankl, A. Kupavskii, \textit{ Two problems of P. Erd\H os on matchings in set families}, submitted, arXiv:1607.06126

\bibitem{FT} P. Frankl, N. Tokushige, \textit{Some best possible inequalities concerning cross-intersecting families}, Journal of Combinatorial Theory, Series A 61 (1992), N1, 87--97.

\bibitem{FT6} P. Frankl, N. Tokushige, {\it A note on Huang–Zhao theorem on intersecting families with large minimum degree}, Discrete Mathematics 340 (2016), N5, 1098--1103.

\bibitem{HK} J. Han, Y. Kohayakawa, {\it The maximum size of a non-trivial intersecting uniform family that is
not a subfamily of the Hilton--Milner family}, Proc. Amer. Math. Soc. 145 (2017) N1, 73--87.
\bibitem{HPS} H. H\'an, Y. Person, M. Schacht, {\it On perfect matchings in uniform hypergraphs with large
minimum vertex degree}, SIAM Journal on Discrete Mathematics 23 (2009), N2, 732--748.

\bibitem{HM} A.J.W. Hilton, E.C. Milner, \textit{Some intersection theorems for systems of finite sets}, Quart. J. Math. Oxford 18 (1967), 369--384.

\bibitem{HZ} H. Huang and Y. Zhao, {\it Degree versions of the Erd\H os-Ko-Rado theorem and Erd\H os hypergraph matching conjecture}, J. Combin. Theory Ser. A, to appear

\bibitem{IK} F. Ihringer, A. Kupavskii, \textit{Regular intersecting families}, preprint.

\bibitem{Ka} G. Katona, \textit{A theorem of finite sets}, ``Theory of Graphs, Proc. Coll. Tihany, 1966'', Akad, Kiado, Budapest, 1968; Classic Papers in Combinatorics (1987), 381-401.

\bibitem{KostM}A. Kostochka, Dhruv Mubayi, {\it The structure of large intersecting families} Proc. Amer. Math. Soc.  145 (2017), N6,  2311--2321.
\bibitem{Kr} J.B. Kruskal, \textit{The Number of Simplices in a Complex}, Mathematical optimization techniques 251 (1963), 251-278.

\bibitem{KO} D. K\"uhn, D. Osthus, {\it Embedding large subgraphs into dense graphs}, Surveys in combinatorics
(2009). Papers from the 22nd British combinatorial conference, St. Andrews, UK, July 5--10, 2009, pages 137--167.

\bibitem{Kup21} A. Kupavskii, {\it Diversity of intersecting families}, submitted
\bibitem{KZ} A. Kupavskii, D. Zakharov, {\it Regular bipartite graphs and intersecting families}, 	arXiv:1611.03129

\end{thebibliography}
\end{document}